\documentclass[oneside,reqno]{amsart}
\usepackage[utf8]{inputenc}
\usepackage[a4paper,top=2.5cm,bottom=3cm,left=2.5cm,right=2.5cm]{geometry}
\usepackage{amsmath,amssymb, mathrsfs}
\usepackage{bbm}
\usepackage{enumerate}
\usepackage{physics}
\usepackage{cite}
\usepackage[dvipsnames]{xcolor}

\usepackage[usenames,dvipsnames]{pstricks}
\usepackage[colorlinks=true, allcolors=BlueViolet, pagebackref=true]{hyperref}

\usepackage{cleveref}

\allowdisplaybreaks[4]

\newtheorem{theorem}{Theorem}[section]
\newtheorem{lemma}{Lemma}[theorem]
\newtheorem{corollary}{Corollary}[theorem]
\newtheorem{proposition}{Proposition}[theorem]

\theoremstyle{definition}

\newtheorem{definition}{Definition}
\newtheorem{remark}{Remark}

\newtheorem{example}{Example}

\newcommand{\LK}{\mathcal{L}}
\newcommand{\R}{\mathbb{R}}
\newcommand{\N}{\mathbb{N}}
\newcommand{\Z}{\mathbb{Z}}
\newcommand{\C}{\mathbb{C}}
\newcommand{\de}{\partial}

\newcommand{\B}{\mathbb{B}}

\DeclareMathOperator{\out}{out}
\DeclareMathOperator{\scal}{Scal}
\DeclareMathOperator{\hess}{Hess}
\DeclareMathOperator{\arcsinh}{arcsinh}
\DeclareMathOperator{\adj}{Adj}

\renewcommand{\a}{\alpha}

\newcommand{\f}{\varphi}
\newcommand{\e}{\varepsilon}
\newcommand{\w}{{\omega}}

\newcommand{\zero}{0}

\newcommand{\h}{\theta}
\newcommand{\ex}{\mathrm e}

\newcommand{\vol}{\mathrm{Vol}}

\newcommand{\hout}[1]{\mathrm{H^{\out}_{#1}}}

\renewcommand{\P}{\mathbb{P}}
\newcommand{\E}{\mathbb{E}}

\newcommand{\mC}{\mathcal{C}}

\newcommand{\G}{K}

\newcommand{\m}[1]{\mathcal{#1}}

\renewcommand*\backref[1]{\ifx#1\relax \else (Cited on p.#1) \fi}
\crefname{equation}{Equation}{Equations}
\crefname{multline}{Equation}{Equations}
\crefname{figure}{Figure}{Figures}
\crefname{subfigure}{Figure}{Figures}
\crefname{question}{Question}{Question}
\crefname{section}{Section}{Sections}
\crefname{subsection}{Subsection}{Subsections}
\crefname{lemma}{Lemma}{Lemmas}
\crefname{proposition}{Proposition}{Propositions}
\crefname{theorem}{Theorem}{Theorems}
\crefname{corollary}{Corollary}{Corollaries}
\crefname{definition}{Definition}{Definitions}
\crefname{remark}{Remark}{Remarks}
\crefname{proposition}{Proposition}{Proposition}
\crefname{corollary}{Corollary}{Corollaries}
\crefname{example}{Example}{Examples}
\crefname{claim}{Claim}{Claim}
\crefname{conjecture}{Conjecture}{Conjecture}

\newcommand{\be}{\begin{equation}}
\newcommand{\ee}{\end{equation}}

\newcommand{\bega}{\begin{equation}\begin{aligned}}
\newcommand{\eega}{\end{aligned}\end{equation}}

\numberwithin{equation}{section}

\usepackage{mathtools} 
\mathtoolsset{showonlyrefs=false,showmanualtags=true} 

\title[Expected LKC for spin random fields]{Expected Lipschitz-Killing curvatures for spin random fields and other non-isotropic fields}
\date{\today}
\author{Francesca Pistolato}
\author{Michele Stecconi}
\thanks{University of Luxembourg, DMATH. \emph{Email addresses:} \texttt{francesca.pistolato@uni.lu} , \texttt{michele.stecconi@uni.lu}}

\begin{document}
\maketitle

\begin{abstract}
Spherical spin random fields are used to model the Cosmic Microwave Background polarization, the study of which is at the heart of modern Cosmology and will be the subject of the LITEBIRD mission, in the 2030s. Its scope is to collect data to test the theoretical predictions of the Cosmic Inflation model.
In particular, the Minkowski functionals, or the Lipschitz-Killing curvatures, of excursion sets can be used to detect deviations from Gaussianity and anisotropies of random fields, being fine descriptors of their geometry and topology.
In this paper, we give an explicit, non-asymptotic, formula for the expectation of the Lipschitz-Killing curvatures of the excursion set of the real part of an arbitrary left-invariant Gaussian spin spherical random field, seen as a field on $SO(3)$. Our findings are coherent with the asymptotic ones presented by Carrón Duque et al. in \emph{Minkowski Functionals in $SO(3)$ for the spin-$2$ CMB polarization field}.
We also give explicit expressions for the Adler-Taylor metric, and its curvature.
We obtain this result as an application of a general formula that applies to any nondegenerate Gaussian random field defined on an arbitrary three-dimensional compact Riemannian manifold. The novelty is that the Lipschitz-Killing curvatures are computed with respect to an arbitrary metric, possibly different from the Adler-Taylor metric of the field.
\end{abstract}

\section{Introduction}
\subsection{Overview}
Our study is closely related to \cite{articolo_dei_fisici_published}, in which asymptotic formulas for the expected Lipschitz-Killing curvatures of the excursion set of Gaussian isotropic spin $s=2$ random fields were derived and probed by numerical simulations. We prove an explicit non-asymptotic formula, valid for fields of arbitrary spin weight (see \cref{subsec:setting}) $s\in \Z$.
In fact, we obtain this formulas as a consequence of a very general result, valid for all Gaussian fields on a three-dimensional Riemannian manifold. 

The relevance of the result is twofold: firstly, spin fields are highly relevant in Cosmology; 
secondly, the general formula is the first instance of a substantial generalization of Adler-Taylor formulas \cite[Theorem 13.4.1]{AdlerTaylor}. See \cref{sec:genAT} below for a thorough discussion and the next paragraph
for a precise account of how this paper complements the results of \cite{articolo_dei_fisici_published}.

\subsection{Motivations}\label{sec:motiv}
\subsubsection{Cosmology}
The topic of spherical spin random fields is strongly connected with the analysis of the Cosmic Microwave Background (CMB), namely, a microwave radiation in which all the observable universe is embedded and that carries information about the early stages of the universe. Its existence was confirmed in 1965, after being predicted in the 40's, see \cite{libro, Dodelson:2003ft}. According to the Standard Cosmological Model \cite{AHeavens_2008}, this radiation is originated by the effects produced by an exponential inflation of the universe (Cosmic Inflation) in the seconds immediately after the Big Bang. Indeed, the measurements of the CMB temperature and its anisotropies played a central role in establishing the latter model. 

The new frontier in this area is the study of CMB polarization \cite{geller2008spin, Cabella:2004mk,Seljak}.
This is a topic of fundamental importance which has gained increasing attention in the last twenty years and it is definitely bound to get more in the near future. Indeed, the LITEBIRD mission, scheduled for 2032, has the scope of collecting measurements and observations of the CMB polarization, which are believed to be a key source of information for probing the existing models and to address the remaining questions, in particular regarding the primordial gravitational waves predicted by the Cosmic Inflation model, see \cite{LiteBIRD_Collaboration2023-jy,CMB2023arXiv231200717C}. 
In particular, see \cite{Nature_Komatsu2022-jv}, understanding the initial fluctuations in the early Universe could shed light on new physical concepts, beyond the Standard Model, required by the Standard Cosmological Model, such as dark matter and dark energy.

Mathematically, the CMB is modeled as the realization of a random spin $2$ field, that is a random section of a complex line bundle over the two-sphere $S^2$. 
An intuitive explanation, taken from \cite[Section 12.1]{libro}, is the following: an experimental recording of the CMB radiation presents as a collection of random ellipses $E_x$ in $T_xS^2$, for each $x\in S^2$. The width of the ellipse is interpreted as the temperature of the CMB, and the two remaining data identifying the ellipse's elongation and orientation form the \emph{CMB polarization}. The former is thus a scalar random field on $S^2$, while the polarization field can essentially be described by a vector field $x\mapsto v(x)\in T_xS^2$ on the sphere, but where $v(x)$ and $-v(x)$ determine the same ellipse, so that the proper object to look at is the spin 2 field $x\mapsto v(x)^{\otimes 2}\in (TS^2)^{\otimes 2}$, see \cite[Section 3.1.1]{GeoSpin2022}.
This model was originally proposed by \cite{geller2008spin}, building on the concept of spin given in \cite{NP66}. As shown in \cite{BR13}, for any spin $s\in \Z$, this model is equivalent to that of a complex-valued field $X$ on $SO(3)$ that satisfies the identity
\begin{equation}\label{eq:Xspins} 
X(pR_3(\psi))=X(p)e^{-is\psi},
\end{equation}
for any $\psi\in\R$ and $p\in SO(3)$. See also \cite{GeoSpin2022,stecconi2021isotropic,malya11}. 
The field $X$, defined in \cref{eq:f} below, has such property. 
In this paper, we study its real part $f=\Re(X)$, motivated by the fact that from a statistical point of view there is no difference between $f$ and $X$, since the identity \eqref{eq:Xspins} implies that the real and imaginary parts of $X$ are completely correlated.

We finally mention that the analysis of spin random fields is also relevant for the study of gravitational lensing, see \cite{libro}.

\subsubsection{Lipschitz-Killing curvatures}
In physics literature, the Lipschitz-Killing curvatures are better known as the Minkowski functionals (they are proportional, see \cite[Equation (6.3.9)]{AdlerTaylor}). They have by now become a standard tool for the study of the morphology of CMB temperature and scalar fields in general, see \cite{Schmalzing1997,MikowMorph,libro,articolo_dei_fisici_published}. As explained in \cite{articolo_dei_fisici_published}, they encode geometric and topological information on the field, not seen in the power spectra, that can detect possible deviations from Gaussianity or statistical isotropy. 

{The analysis of the  Lipschitz-Killing curvatures of the excursion set of Gaussian fields has been the object of a vast literature, \cite{ASTexcSta, MariWig2011,BB12,CamMarWig16,CammarotaM2018,kratzVadlamani,BDiBDE19,VidLipschitz,Camma23,Kuriki_Matsubara_2023} among others, a pillar of which are the formulas by Adler and Taylor \cite{AdlerTaylor} for computing the expectation.}
Normally, the underlying assumptions imply that the field is somehow \emph{isotropic} (see \cite[pp. 115, 324]{AdlerTaylor} and \cref{subsec:leftVSiso}), in the sense that the geometry induced by it should coincide (up to a constant factor) with that of the manifold itself (see \cref{sec:homo} for details). 
The spin fields, however, do not have this property, and because of this, the standard formulas are not sufficiently sophisticated (see \cref{sec:genAT}). For this reason, in this paper, we prove a generalization of Adler and Taylor's formulas (see \cref{thm:main2}), in dimension three, without making any assumption on the geometry induced by the field.

\subsection{Setting and main results}\label{subsec:setting}
We work under the tacit assumption that all random variables that will appear from now on, are defined on a common probability space $(\Omega,\mathcal F, \P)$. We denote by $\E$ the expectation, i.e., the integral with respect to $\P$.

We denote by $SO(3)$ the group of real $3\times3$ orthogonal matrices with determinant $+1$, that is $SO(3)=\{\,R\in\R^{3\times3}: R^{\!T}R=I_3,\ \det R=1\,\}$. As a manifold, $SO(3)$ is a compact and connected Lie group of dimension $3$. We endow $SO(3)$ with the left-invariant Riemannian metric induced by its embedding into $\R^6$, sending a matrix to the vector formed by its last two columns; see \cref{sec:geomSO3}.

In this paper, we focus our attention on the class of real-valued smooth Gaussian fields $\{ f(p):p\in SO(3)\}$ on $SO(3)$ that are of the form
\begin{equation}\label{eq:f}
f=\Re(X),\quad \text{where} \quad X= \sum_{l=|s|}^{\infty}c_l\sum_{m=-l}^l \gamma^l_{m,s} D^l_{m,s},
\end{equation}
with $\gamma^l_{m,s}$ being i.i.d. complex Gaussian random variables\footnote{So that $\Re \gamma ^l_{ms},\Im \gamma ^l_{ms}\sim \m N(0,\frac{1}{2})$ are independent.}, $D^l_{m,s}:SO(3)\to \C$ are the coefficients of Wigner matrices (see \cite[Section 3.3]{libro}), $c_l>0$ are real positive constants\footnote{The sequence $C_l=c_l^2\frac{4\pi}{2\ell+1}$ is called the \emph{angular power spectrum}. The decay rate of the sequence $c_l$ should be such that the series converges in $\mC^\infty(SO(3))$.}, and $\Re(z)$, $z\in\mathbb C$ denotes the real part of the complex number $z$. In this context, the number $s\in \Z$ is called the \emph{spin weight} and we say that $f$ is a real field of \emph{spin $s$} cf. \cite{NP66,geller2008spin,malya11,stecconi2021isotropic}\footnote{There is no uniformity in the literature regarding the choice of the sign of $s$, meaning that some authors would call $f$ a field of spin $-s$. In this paper, we take the same convention as in \cite{NP66,GeoSpin2022,stecconi2021isotropic,libro,articolo_dei_fisici_published}
See \cite[page 1085]{malya11} and the references therein for an account of different conventions.}. In the following, we denote by $k$ the circular covariance function of the complex field $X$, such that \begin{equation}\label{eq:circulark}
    k(\theta) 
    =\sum_{l=|s|}^\infty c_l^2 d^l_{s,s}(\theta), \quad \theta \in \R,
\end{equation}
where $\{d_{m,s}^l\}_{m,s,l}$ are Wigner d-functions; this object encodes the covariance functions of both $X$ and $f$, see \cref{subsec:spinrfs}, \cref{eq:k}.
\begin{remark}
[On the regularity of $f$] \label{rem:onthereg} We stress that throughout this paper, we will assume the field $f$ to be a smooth function on $SO(3)$, with probability one. In terms of the representation given in \cref{eq:f} this translates into the requirement that the coefficients $c_l$ decay sufficiently fast; precisely, we assume that the sequence $(c_l)_{l\ge |s|}$ is such that the series in \cref{eq:circulark}, above, converges in the $\mC^\infty(\R)$ topology as a function of $\theta$. Note that if all but a finite number of the coefficients $c_l$ are zero, then $f$ is smooth. In this paper, we will not discuss any other criteria to verify such smoothness condition.
\end{remark}
Our first main result --- \cref{thm:main1} --- is an explicit formula for 
\begin{equation} \label{eq:introESO}
\E\left\{ \LK_i\left( f\ge u\right) \right\}:=\E\left\{ \LK_i\left( A_u(f)\right) \right\},
\end{equation}
i.e., the expectation of the \emph{Lipschitz-Killing curvatures} $\LK_i$ (see \cite[Part II]{AdlerTaylor}, or \cref{sec:LK3d} below), for $i\in \{0,1,2,3\}$, of the excursion set of $f$, that is the random set 
\begin{equation}\label{eq:introAu}
    A_u(f):=\{p\in SO(3)\colon f(p)\ge u\},
\end{equation} 
for any deterministic value of $u\in \R$. A fact that is by now standard in the literature (see, for instance, \cite[Corollary 11.3.3]{AdlerTaylor}, or \cref{lem:smoothAu} below) is that, if $f$ has positive variance, the subset $A_u(f)\subset SO(3)$ is almost surely a three-dimensional submanifold with smooth boundary $\de A_u=f^{-1}(u)$. 

Given a compact Riemannian manifold with boundary, $(A,g)$, we denote by $\LK_i^g$ the associated $i^{th}$ Lipschitz-Killing curvature, $\m H^i_g$ the associated $i$-dimensional Hausdorff measure and by $\vol_g=\LK_{\dim A}^g=\m{H}^{\dim A}_g$ the volume measure.
If $\dim M=3$, we have that: $\LK_3^g(A)$ is the volume of $A$; $\LK_2^g\left( A\right)$ is (one half) the surface area of the boundary $\de A$; $\LK_0^g\left( A\right)$ is the Euler-Poincaré characteristic of $A$; and 
\begin{equation}\label{eq:introL1}
\LK_1^g\left( A\right)=-\frac{1}{\pi} \int_{\de A} \hout{\de A} d\m H^2_g  +  \frac{1}{4\pi}\int_{A} \scal_g d\vol_g,
\end{equation} 
where $\hout{\de A}$ denotes the mean curvature of $\de A$ in the outer direction and $\scal_g$ denotes the scalar curvature (in \cref{sec:LK3d}, we derive these formulas from the general one).
The Lipschitz-Killing curvatures in \eqref{eq:introESO} are meant with respect to the Riemannian metric of $SO(3)$ determined by its identification with the subsets of $\R^6$ consisting of orthonormal pairs $(v,p)$ of vectors in $\R^3$ (i.e., the unit tangent bundle of $S^2$, see \cref{sec:geomSO3}). For the objects relative to this metric on $SO(3)$, we write $\LK_i$, $\vol$, $\m H^i$, without superscripts or subscripts for the metric. When clarity is needed, we denote by $g_{\R^6}$ the metric on $SO(3)$.

{By \emph{explicit formula} we mean an expression in terms of $u$, $s$, the coefficients $c_l$, and the \emph{Lipschitz-Killing densities} $\Xi_i$, $i=0,1,2,3$, that do not involve integrals, except for the one implicit in the special function $\Phi(u)$. 
Our findings are in accordance with the asymptotic formulas obtained in the recent work \cite{articolo_dei_fisici_published}.
}
{
\begin{theorem}\label{thm:main1}
Let $f:SO(3)\to \R$ be a random field defined as above, having spin $s\in \Z$ and with
\begin{equation}\label{eq:intro_norm}
1=\E\left[|f(p)|^2\right]=\sum_{l=|s|}^\infty 
\frac{c_l^2}{2}=\frac{k(0)}{2}, 
\quad \xi^2:=\sum_{l=|s|}^\infty 
\frac{c_l^2}2\frac{(l(l+1)-s^2)}{2}=-\frac{k''(0)}{2}. 
\end{equation} 
Then, for every $u\in \R$, and $j=0,1,2,3$, we can write \begin{equation}\label{eq:LKdensities}
    \E[\mathcal L_j (f\ge u)] = \sum_{i=0}^{n-j} \mathcal L_{i+j}(SO(3)) \ \Xi_i(u),
\end{equation}
where the functions $\Xi_i$, $i=0,1,2,3$ depends on the threshold $u$, are defined by
    \begin{align*}
    &\Xi_0(u) = \left(1-\Phi(u)\right), \\
    &\Xi_1(u) = e^{-u^2/2} d_0(\xi,s), \\
    &\Xi_2(u) = ue^{-u^2/2} d_1(\xi,s), \\
    &\Xi_3(u) = (u^2-1)e^{-u^2/2} d_2(\xi,s) + e^{-u^2/2} d_3(\xi,s),
\end{align*}
and $\Phi(u)=\int_{-\infty}^u(2\pi)^{-\frac12}\exp(-\frac{t^2}{2})dt$ denotes the cumulative distribution function of a normal Gaussian $\m N(0,1)$. The $d_i$'s are constants depending on $\xi$ and $s$, that are defined, for any $\xi>|s|$, by \begin{align*}
    d_0(\xi,s) & := \frac1{2\pi}   \left( \xi \frac{\arcsin{\sqrt{1-\frac{s^2}{\xi^2}}}}{\sqrt{1-\frac{s^2}{\xi^2}}} +  |s|\right),\\
    d_1(\xi,s) & := \frac{1}{\sqrt{8\pi^2}} \Bigg[ \xi^2 + \frac{s^2}{\sqrt{1-\frac{s^2}{\xi^2}}} \log \xi \\
    &\hspace{2cm}-s^2 \left(  \frac{\log |s|}{\sqrt{1-\frac{s^2}{\xi^2}}}  - \frac1{2\sqrt{1-\frac{s^2}{\xi^2}}} \log \left( 1 + \sqrt{1-\frac{s^2}{\xi^2}} \right) \right)\Bigg] ,\\
    d_2(\xi,s) & := \frac1{4\pi^2}|s|\xi^2,\\
    d_3(\xi,s) & := \frac1{4\pi^2}|s| \left(1-\frac{s^2}{4\xi^2}\right) - \frac3{8\pi} d_0(\xi,s).
\end{align*}
\end{theorem}
}
\begin{remark}
    For the sake of brevity, we state the formulas for $\xi>|s|$. We refer to \ref{appB} for those when $\xi< |s|$.
\end{remark}
\begin{remark}
    If $\xi = |s|$, the field is $f$ homothetic. Therefore, $\E \LK_j(f\ge u)$, $j=0,1,2,3$, $u\in\R$, can be computed applying Adler-Taylor formulas \cref{eq:ATformula}, since $\E \LK_j(f\ge u) = \xi^{j} \, \E \LK_j^f(f\ge u)$. We remark that the same formulas can be obtained applying \cref{thm:main2}, with $a_1=a_2=a_3\equiv\xi^2$. See also \cref{sec:homo}.
\end{remark}

We obtain \cref{thm:main1} as the specialization of a more general formula for \eqref{eq:introESO}, valid for arbitrary Gaussian random fields on an arbitrary three-dimensional Riemannian manifold $(M,g)$. This is the content of \cref{thm:main2}, the second main result of the paper. 
\cref{thm:main2} reduces the computation of \eqref{eq:introESO} to that of certain invariants of a Riemannian metric associated to $f$.
Any non-degenerate smooth Gaussian field $f=\{f(p):p\in M\}$ has an associated Riemannian metric $g^f$, defined by \begin{equation} \label{eq:ATmetric} 
g^f_p(v,v)=\E|d_pf(v)|^2 \ \forall p\in M \text{ and } v\in T_pM,
\end{equation}
where the non-degeneracy means precisely that $g^f$ is a metric.
We refer to $g^f$ as the \emph{Adler-Taylor metric of $f$} (see \cite[Section 12.2]{AdlerTaylor}).
Its main properties will be recalled in \cref{sec:AT}.
The two metrics $g$ and $g^f$ might differ --- indeed, they do in the case of \cref{thm:main1} (see \cref{eq:metric}) --- and we compare them in terms of the eigenvalues of one with respect to the other. These are 
$d=\dim M$ real-valued positive functions $a_1,\dots,a_d$ on $M$ such that the matrix of $g^f_p$ in any orthonormal basis of $(T_pM,g_p)$ has eigenvalues $a_1(p),\dots,a_d(p)$.
\footnote{In fact, the spectral theorem ensures that there exists an orthonormal basis of $g_x$ such that the relative matrix of $g^f_p$ is diagonal.}
We show that the value of $\E[\mathcal L_3^g (f\ge u)]$ depends solely on the two metrics $g$ and $g^f$.
\begin{theorem}\label{thm:main2}
Let $f=\{f(p):p\in M\}$ be a real-valued smooth Gaussian random field defined over a three-dimensional compact smooth Riemannian manifold $(M,g)$. Assume that $f$ has unit variance and that the Adler-Taylor metric $g^f_p=\E\left\{ d_pf^{\otimes 2}\right\}$ of $f$ has strictly positive eigenvalues $a_1(p)$, $a_2(p)$, $a_3(p)$ with respect to $g_p$, at any point $p\in M$. Then,
\begin{align*}
& \E[\mathcal L_3^g (f\ge u)] =\vol_g(M)(1-\Phi(u)), \\
& \E[\mathcal L_2^g (f\ge u)] = \frac{1}{\sqrt{8\pi}} \mathrm e^{-u^2/2} \int_M E_2(a_1,a_2,a_3) d\vol_g,\\
&\E[\mathcal L_1^g (f\ge u)] = \frac1{\sqrt{8\pi^3}} u e^{-u^2/2} \int_M \left( a_1+a_2+a_3 - E_1(a_1,a_2,a_3)\right)d\vol_g \\&\hspace{5cm} 
+ \left( 1-\Phi(u)\right) \mathcal L_1^g(M),\\
&\E[\mathcal L_0^g (f\ge u)] =\frac{e^{-\frac{u^2}{2}}}{4\pi^2} \int_{M}
\left((u^2-1)+\frac{1}{2}\scal^f \right)\sqrt{a_1a_2a_3} \ d \vol_g.
\end{align*}
Here, $\scal^f$ denotes the scalar curvature of $(M,g^f)$ and $E_1$, $E_2$ are as in \cref{def:E1E2}. 
\end{theorem}
We highlight that the formulas for $\E[\mathcal L_2^g (f\ge u)]$ and $\E[\mathcal L_1^g (f\ge u)]$ are new: they are not deductible from \cite[Theorem 13.4.1]{AdlerTaylor}, that correspond to the special case $a_1=a_2=a_3=1$ (we will discuss this point in more detail in \cref{sec:genAT}). See also \cref{rem:euclidean} for the example of a stationary field in a Euclidean setting ($M=\R^3$, and constant eigenvalues $a_1$, $a_2$, $a_3$).
The dependence on $f$ of the right hand sides of the former new formulas is only through the functions $E_1$ and $E_2$, comparing the two metrics $g^f$ and $g$, which are defined as follows. 
\begin{definition}\label{def:E1E2}
Let $a_1,a_2,a_3 \in (0,+\infty)$. Let $\gamma_1,\gamma_2,\gamma_3 \sim \m N(0,1)$ be independent real normal variables. Then, we define
\begin{equation} 
E_1(a_1,a_2,a_3):=\E\left\{ \frac{\sum_{i=1}^3 a_i^2\gamma_i^2}{\sum_{i=1}^3 a_i\gamma_i^2}\right\}, \ \text{ and } \ 
E_2(a_1,a_2,a_3):=\E\left\{ \sqrt{\sum_{i=1}^3 a_i\gamma_i^2}\right\}.
\end{equation}
\end{definition}
In the case of \cref{thm:main1}, that is, when we consider $f$ as in \eqref{eq:f}, we are able to compute explicitly all the needed invariants. In particular, in a specific choice of coordinates of $SO(3)$, we have the following result.
\begin{theorem}\label{lem:metric}
Let us consider $f$ as in \cref{thm:main1}, with circular covariance function $k$, as defined in \eqref{eq:circulark}. Then, $f$ is a smooth Gaussian field on $M=SO(3)$. The Gram matrix of the Adler-Taylor metric $g^f$ at a point $p\in SO(3)$, in the Euler angles coordinates $p=R(\f,\theta,\psi)$, see \cref{def:eulerangles}, is
\begin{equation}\label{eq:metric}
 \Sigma_{(\xi,s)}(\theta) 
      =  \begin{pmatrix}
         \xi^2 \sin^2(\h)+s^2 \cos^2(\h) & \zero & s^2 \cos(\h) \\
         \zero & \xi^2 & \zero \\
         s^2 \cos(\h) & \zero & s^2
     \end{pmatrix}.
\end{equation}
The standard metric on $SO(3)$ has Gram matrix $\Sigma_{1,1}(\h)$, so for all $p\in SO(3)$ we have
\begin{equation}
    \label{eq:abc}
a_1(p),a_2(p),a_3(p)= \xi^2,\xi^2,s^2.
\end{equation}  
Moreover, the scalar curvature of the metric $g^f$ is constant and equals \begin{equation}\label{eq:scalf}
\scal^f = \frac{2}{\xi^2} - \frac{s^2}{2\xi^4}.
\end{equation}
\end{theorem}
In this case, when two eigenvalues coincide, we compute the expectation and give an explicit formula for the functions $E_1$ and $E_2$, from which we deduce \cref{thm:main1}.
\begin{proposition}\label{thm:E1E2}
The functions $E_1,E_2$ of \cref{def:E1E2} satisfy the following identities, for any value of $\xi>0$ and $s\neq 0$ such that $|\xi|>|s|$:
\begin{align}
E_1(\xi^2,\xi^2,s^2)
&=    \xi^2 \Bigg\{ 1 + \frac{s^2}{\xi^2}\frac1{\sqrt{1-\frac{s^2}{\xi^2}}} \log \frac{|s|}{\xi}   \nonumber
 \\ & \hspace{2cm} 
 + \frac{s^2}{\xi^2} \Bigg(  1 - \frac1{2\sqrt{1-\frac{s^2}{\xi^2}}} \log \left( 1 + \sqrt{1-\frac{s^2}{\xi^2}}  \right) \Bigg)\Bigg\}; \\
E_2(\xi^2,\xi^2,s^2)&=\xi\sqrt{\frac2{\pi}}\frac{\arcsin{\sqrt{1-\frac{s^2}{\xi^2}}}}{\sqrt{1-\frac{s^2}{\xi^2}}} + |s| \sqrt{\frac2\pi}.
\end{align}
\end{proposition}

\subsection{Structure of the paper}
In \cref{sec:remarks}, we collect a list of remarks. In \cref{sec:prel}, we present preliminary definitions and results. In \cref{sec:proofs}, we provide, in the following order, the proofs of the main results: \cref{thm:main2}, \cref{thm:main1}, and \cref{lem:metric}, except for the computation of the scalar curvature \cref{eq:scal}, which is in \ref{appA}. In the latter, we include the detailed computations of the Riemann tensor, the sectional curvatures, and the Lipschitz-Killing curvatures of $SO(3)$ in the Adler-Taylor metric. In \ref{appB} and \ref{appC}, we include, respectively, the proof of \cref{thm:E1E2} and an explanation of the formula for the Lipschitz-Killing curvatures stated in the Preliminaries.

\section{Remarks}\label{sec:remarks}
\subsection{Main novelties}\label{sec:novelty} 
\subsubsection{Comparison with \cite{articolo_dei_fisici_published}}
The original motivation for our \cref{thm:main1} is to provide a \emph{static}, i.e. non-asymptotic, version of the formulas obtained in \cite{articolo_dei_fisici_published}, where the same problem is tackled for spin $2$ fields (they model the CMB, see \cref{sec:motiv}) and in the limit $\xi\to +\infty$. 

The metric on $SO(3)$ used in \cite[(2.2)]{articolo_dei_fisici_published} is $g_{\R^9}=2g_{\R^6}=:2g$. 
Consequently, all Lipschitz curvatures $\mathcal{L}_j^{2g}$ computed in \cite{articolo_dei_fisici_published} differ by a power of $2$ with respect to ours:
\begin{equation} 
\mathcal L_j=\mathcal L_j^g={2^{-\frac{j}{2}}}\mathcal{L}_j^{2g}.
\end{equation}
Moreover, we set $\mu:=\frac{1}{5}\xi^2|s|$, so that  \cite[Equation (3.7)]{articolo_dei_fisici_published} is satisfied:
\begin{equation} 
 \frac{5}{2\sqrt{2}}\mu=\sqrt{\det \left(\Sigma\right)}=2^{-\frac 32}\xi^2|s|.
\end{equation}
Here, $\Sigma$ is the Gram matrix of the metric $g^f$, in an orthonormal basis with respect to the metric $2g$, see \cite[Eq. (3.7)]{articolo_dei_fisici_published}, so that $2\Sigma=\Sigma_{\xi,s}$ defined in \eqref{eq:Sigmaxis}. In order to make the comparison with \cite{articolo_dei_fisici_published} easier, we write the asymptotics derived from  Theorem~\ref{thm:main1} as function of $\mu$.
\begin{corollary}\label{cor:main} 
Let $f$ be as above. We have the following asymptotic behavior as $\mu\to +\infty$:
\begin{equation}\label{eq:asy}
\begin{aligned}
        \E \mathcal L_3^{2g}(f\ge u) &= 2^{\frac{3}{2}}\cdot 8\pi^2\left(1-\Phi(u)\right) ;
\\
\E \mathcal L_2^{2g}(f\ge u) &=
{ 2\cdot 2 \pi^2 e^{-u^2/2} \sqrt{\frac{5}{|s|}} \sqrt{\mu} \left( 1 + o(1)\right)};
\\
\E \mathcal L_1^{2g}(f\ge u) &= 
{ 2^{\frac12}\cdot 2^{\frac52}5 \sqrt \pi u e^{-u^2/2} \frac1{|s|} \mu \left( 1 + o(1)\right)};
\\
        \E \mathcal L_0^{2g}(f\ge u) &= 10 (u^2-1) e^{-u^2/2}\mu+O(1). 
\end{aligned}
\end{equation}
\end{corollary}

\begin{remark}
Our formulas \eqref{eq:asy} for $\E \LK_3$ and for $\E \LK_0$ differ from the ones in \cite{articolo_dei_fisici_published} by the same constant factor $2^{\frac{5}{2}}$.
\end{remark}
\begin{remark}
The asymptotic formulas \eqref{eq:asy} for $\E \LK_2$ and $\E\LK_1$ are only given up to constant factors $K_1$ and $K_2$, in \cite{articolo_dei_fisici_published}, which can now be deduced from \cref{cor:main}.
Moreover, their derivation is subordinate to the conjecture that $\LK_1^f=K_2^{-1}\LK_1$ and $\LK_2^f=K_1^{-1}\mu^{\frac12}\LK_2$. From \cref{cor:main}, we can see that this conjecture holds true in expectation and asymptotically.
\end{remark}

\subsubsection{Generalization of Adler and Taylor formulae}\label{sec:genAT}
Consider the setting of \cref{thm:main2}.
In what follows, we will use the superscript $f$ 
for Riemannian quantities, to denote that they are computed with respect to the metric $g^f$ associated with $f$. No superscript means that the quantity is computed with respect to the original metric $g$. The Adler-Taylor formula \cite[Theorem 13.4.1]{AdlerTaylor} for Lipschitz-Killing curvatures states:
\begin{equation}\label{eq:ATformula}
\E \, \mathcal{L}^{f}_i(\{f\ge u\})=\sum_{0\le j\le 3-i}\frac{\w_{i+j}}{\w_j\w_i}\binom{i+j}{j}
\mathcal{L}_{i+j}^{f}(M)\rho_j(u),
\end{equation}
where $\w_d={\pi^{\frac d2}}{\Gamma\left( \frac{d}{2}+1\right)}^{-1}$ is the $d$-volume of the unit ball in $\R^d$, $\rho_0(u)=1-\Phi(u)$ is the tail probability function for a standard Gaussian variable, and \begin{equation}
\rho_j(u)=(2\pi)^{-\frac{j+1}{2}}H_{j-1}(u)e^{-\frac{u^2}{2}}, \quad j\ge 1    ,
\end{equation}
cf. \cite[Equation (12.4.2)]{AdlerTaylor}.
Note that the Lipschitz--Killing curvatures in both sides are computed with respect to the metric $g^f$, hence the above formula does not say much about 
\begin{equation}\label{eq:q}
\E \, \mathcal{L}_j(\{f\ge u\})=\ ?,
\end{equation}
where $\mathcal{L}_j$ is relative to the metric $g$ on $M$, and no relation between $g$ and $f$ is assumed a priori. 
\cref{thm:main2} provides a formula for \eqref{eq:q} in dimension three, generalizing \cref{eq:ATformula} to a setting where no relation between $g$ and $f$ is assumed.
\begin{remark}
The question \eqref{eq:q} has a quick answer, for $j=0$ and $j=\dim M$. 
Indeed, $\E[\mathcal L_0^g (f\ge u)]$ is deduced from \cref{eq:ATformula}, in virtue of the fact that $\LK_0^g$, being a topological quantity, does not depend on the metric $g$. The formula for $\E[\mathcal L_{\dim M}^g (f\ge u)]$ follows from a direct application of Tonelli's theorem. Moreover, since in dimension three the second Lipschitz-Killing curvature $\LK_2$ coincides with the boundary area, we deduce the formula for $\E[\mathcal L_2^g (f\ge u)]$ from \cite[Theorem 6.8]{AzaisWscheborbook}. The most challenging and interesting case is that of $\LK_1$. We prove the formula for $\E[\mathcal L_1^g (f\ge u)]$ by reducing the problem to a suitable form of the Kac-Rice formula \cite[Theorem 6.2]{MathiStec} (a reformulation of \cite[Theorem 6.10]{AzaisWscheborbook}).
\end{remark}
\begin{remark}
Notice that $\E[\mathcal L_1^g (f\ge u)]$ does not involve derivatives of the functions $a_1,a_2,a_3$, but instead it depends only on the point-wise comparison between the two metrics. In other words, it does not involve curvature terms of $g^f$, despite the random variable $\LK_1^g (f\ge u)$ depends on the curvature of the surface $\{f=u\}$, see  \cref{eq:introL1}. 
\end{remark}

\subsubsection{Non-homothetic fields}\label{sec:homo}
The most studied examples of Gaussian random fields have the property that $g^f$ and $g$ are homothetic, that is,
\begin{equation}\label{eq:conformal} 
g^f=\xi^2 g,
\end{equation} 
for some constant $\xi>0$, which implies that $\mathcal{L}_i^{f}\, =\, \xi^i\mathcal{L}_i$; hence, 
the formula \eqref{eq:ATformula} is sufficient, and it is a standard tool. Specifically, when $a_1=a_2=a_3=1$, \cref{thm:main2} reduces to \cref{eq:ATformula}. This is the case of stationary and isotropic fields on $\R^d$ (see \cite[Eq. (5.7.3)]{AdlerTaylor}) or on the torus  $\mathbb T^d=\R^d/ \Z^d$\footnote{As a consequence of its flatness, see \cite[Subsec. (12.2.3)]{AdlerTaylor}.}, random spherical Laplace eigenfunctions (see \cite{MariWig2011,CammarotaM2018}) and arithmetic random waves (see \cite{RudnickWigmanTorus,Camma23}), all isotropic Gaussian fields on spheres $S^d$\footnote{As a consequence of isotropy, see \cite[p. 324]{AdlerTaylor}.}, and, in general, it is the case of Gaussian fields that are invariant (strictly left-invariant \cite[Definition 2.5]{Malyarenko1999-ay} or strongly-isotropic \cite[Definition 2.5]{libro}) under a large group of isometries. The aforementioned fields are invariant under, respectively, the group of rigid transformations, the group of isometries of the torus, and the orthogonal group $O(n)$. A sufficient condition, valid in all the cases above, is that the group acts transitively on the tangent bundle.

On the other hand, the condition \eqref{eq:conformal} is actually very special and leaves out, in particular, the important case of general Riemannian random waves (see \cite{zelditch_2009_rczRRW,Canzani2020-ty,RiviereDang2018-gj,SarnakWigman2019}). For a generic Riemannian manifold, the identity \eqref{eq:conformal} is false, although it
holds asymptotically in the high-frequency limit $\xi\to +\infty$ up to an $o(\xi^2)$ term, provided that the manifold is either Zoll, aperiodic (see \cite[Proposition 2.3]{zelditch_2009_rczRRW}), or a manifold of isotropic scaling (see \cite[Definition 1]{Canzani2020-ty}).

In full generality, in the setting of \cref{thm:main2}, the Adler--Taylor metric $g^f$ may be any other Riemannian metric on $M$, by virtue of Nash's Isometric Embedding Theorem, see \cite[Remark 6.4]{MathiStec} and the discussion at \cite[page 329]{AdlerTaylor}. 

\subsubsection{The case of spin fields on $SO(3)$}
In the case of the field $f$ on $SO(3)$, defined in \cref{eq:f}, the Adler-Taylor metric $g^f$ has constant eigenvalues $\xi^2,\xi^2,s^2$ with respect to the standard metric 
of $SO(3)$ (see \cref{eq:Sigmaxis}); thus, the identity \eqref{eq:conformal} holds only when $\xi=|s|$. 
Therefore, \cref{thm:main1} cannot be proven using Adler-Taylor formulas \eqref{eq:ATformula}, but it follows from \cref{thm:main2} and \cref{thm:E1E2}. 
\begin{remark}
To the best of our knowledge, this is the first time in which all $\E\{\LK_i(f\ge u)\}$ have been explicitly computed for a field $f$ on a Riemannian manifold $(M,g)$ that does not satisfy the condition \eqref{eq:conformal}. 
\end{remark}

\subsubsection{Euclidean setting}\label{rem:euclidean} 
Let us consider the Euclidean case $\R^3$ and let $F$ be a stationary smooth Gaussian field. Then, its Adler-Taylor metric $g^F$ has constant eigenvalues $a_1$, $a_2$, and $a_3$. 
This setting was recently studied in the work \cite{bierme25} of Biermé and Desolneux. In this case, \cref{thm:main2} takes the following form, which is equivalent to \cite[Theorem 4.1]{bierme25}. For $U\subset \R^3$ open, $\vol{}(U)=\vol^g(U)=|U|$ denotes the Lebesgue measure of $U$. 
\begin{corollary}\label{cor:main2Eucli}
Let $F=\{F(p):p\in \R^3\}$ be a real-valued smooth stationary Gaussian field, let $f=F|_U$ be its restriction to an open bounded set $U\subset \R^3$ and let $g$ be the standard Euclidean metric. Assume that $f$ is centered and unit-variance. 
Then,
\begin{align*}
& \E[\mathcal L_3^g (f\ge u)] =|U|\,(1-\Phi(u)), \\
& \E[\mathcal L_2^g (f\ge u)] = \frac{1}{\sqrt{8\pi}}|U| \, \mathrm e^{-u^2/2} E_2(a_1,a_2,a_3) ,\\
&\E[\mathcal L_1^g (f\ge u)] = \frac1{\sqrt{8\pi^3}} |U| \, u e^{-u^2/2} \left( a_1+a_2+a_3 - E_1(a_1,a_2,a_3)\right), \\
&\E[\mathcal L_0^g (f\ge u)] =\frac{e^{-\frac{u^2}{2}}}{4\pi^2} |U| \, 
(u^2-1)\sqrt{a_1a_2a_3}.
\end{align*}
Here, $E_1$, $E_2$ are as in \cref{def:E1E2}.
\end{corollary}  
\begin{proof}
For $j=1,2,3$, the proof of \cref{thm:main2} extends verbatim to the computation of the Lipschitz-Killing curvature \emph{measure} (see \cref{sec:LK3d}) on any bounded open set $U$. Namely,
\begin{align}
\E[\mathcal L_j^g(f\ge u)]=\E\left[\mathcal L_j^g \left(\{F\ge u\}, \{F\ge u\}\cap U\right)\right]
=\int_{U}\rho_j(t)dt,
\end{align}
where $\rho_j(t)$ are exactly the integrands in \cref{thm:main2}. By stationarity, they must be constant, which yields the formulas stated. 
The case $j=0$ is slightly different, because the proof relied on the characterization of $\m L_0$ as the Euler characteristic, which does not apply in the non-compact case. In this case, we refer to \cite[Theorem 4.1]{bierme25} for a proof of the formula in the statement.
\end{proof}
\begin{remark}
Assuming in addition that $f$ is almost surely $\Z^3$-periodic, then $f$ descends to a smooth Gaussian field on the three-torus $\R^3/\Z^3\cong (S^1)^3$, with the same associated eigenvalues. 
In such case, and taking $U=(-1,1)^3$, \cref{cor:main2Eucli} follows by a direct application of \cref{thm:main2} to the case $M=(S^1)^3$.
\end{remark}
In \cite[Theorem 4.1]{bierme25}, the authors consider a stationary Gaussian field $X$, defined on $\R^3$, having mean $\mu$ and variance $\sigma^2$. They compute \begin{equation}\label{eq:previousform}
    \overline{C}_j^X(t) := \frac{\E[\LK_j((X\ge t)\cap U)]}{|U|}, \qquad j=0,1,2,3,
\end{equation}
where $U\subset \R^3$ is a bounded open set, in terms of three constants $\overline{\gamma}_{SA}$, $\overline{\gamma}_{TMC}$, and $\overline{\gamma}_{TGC}$. We shall write $\E[\LK_j((X\ge t)\cap U)] = \E[\LK_j((f\ge \frac{t-\mu}{\sigma}))]$, for a centered and unit-variance field $f : = \big[\frac{X-\mu}{\sigma}\big]\big|_{ U}$. 
In this case, 
the numerator in \cref{eq:previousform} can be computed in terms of $a_1$, $a_2$ and $a_3$, applying \cref{cor:main2Eucli}. The equivalence between \cite[Theorem 4.1]{bierme25} and \cref{cor:main2Eucli} is due to the following identities:
\begin{align}
    \overline{\gamma}_{SA} & = \frac8{\pi} \, E_2(a_1,a_2,a_3)^2 , \\
    \overline{\gamma}_{TMC} & = \frac12 \left( a_1+a_2+a_3 - E_1(a_1,a_2,a_3)\right), \\
    \overline{\gamma}_{TGC} & = (a_1a_2a_3)^{1/3}.
\end{align}
These relations follow from straightforward algebraic manipulation of the definitions of the constants involved.

\subsection{Further remarks}
\subsubsection{Left-invariant metrics}\label{subsec:leftVSiso}
\cref{lem:metric} entails that the matrix of $g^f$ in the coordinates given by the Euler angles $\f,\theta,\psi$ 
(see \cite[Proposition 3.1]{libro}, and \cref{def:eulerangles} below), at the point $(0,\frac{\pi}{2},0)$ is:
\begin{equation}\label{eq:Sigmaxis}
\Sigma_{(\xi,s)}:=
\begin{pmatrix}
\xi^2 & \zero & 0 \\ \zero & \xi^2 & \zero \\
0 & \zero & s^2
\end{pmatrix}.
\end{equation}
It is straightforward to see that for any choice of $\xi>0$ and $s\neq 0$ (regardless of whether they come from a Gaussian field), there exists one and only one left-invariant Riemannian metric on $SO(3)$ with such local expression. We will denote it by $g_{\xi,s}$.

Since the spin field $f$ is \emph{left-invariant}, its associated metric $g^f$ is left-invariant as well; hence, $g_{\xi,s}=g^f$ is the Adler-Taylor metric of $f$, see \cref{lem:metric}.
Moreover, $f$ is invariant under all isometries of $SO(3)$ if and only if it is also right-invariant. This happens precisely when \cref{eq:conformal} holds, as a consequence of the following.
\begin{remark}\label{pr:leftrightinvariance}
 The standard left-right-invariant metric of $SO(3)$ is $g_{1,1}$. All other metrics $g_{\xi,s}$ are left-invariant for any choice of $\xi>0$, $s\neq 0$, and right-invariant if and only if $\xi=|s|$. This is an easy consequence of \cref{lem:geometrySO3}.
 \end{remark}

\cref{lem:metric} and all computations in \ref{appA} continue to hold for the Gram matrix, the Riemann tensor, the scalar curvature, the sectional curvature, and the Lipschitz-Killing curvatures of the Riemannian manifold $(SO(3),g_{\xi,s})$, for any pair $\xi\neq 0$, $s\neq 0$.

\subsubsection{Non-spin fields}\label{sec:nonspin}
Observe that not all metrics $g_{\xi,s}$ come from a spin field defined as in \cref{eq:f}: for instance, if $s\notin \Z$. An example is a linear combination $f=a_1f_1+\dots + a_kf_k$ of independent fields $f_k$, each defined as in \cref{eq:f} with different spin weights $s_k$ and $a_1^2+\dots+a_k^2=1$; then, $g^f=\sum_{i=1}^ka_i^2g_{\xi_i,s_i}$. As a consequence of the stochastic Peter-Weyl Theorem \cite[Theorem 5.5]{libro}, any left-invariant smooth Gaussian field with unit variance must be of the previous kind, possibly with $k=\infty$.
Applying \cref{thm:main2}, we have the following extension of \cref{thm:main1}.
\begin{corollary}\label{cor:xies}
Let $\xi>|s|>0$. Let $f=\{f(p)\}_{p\in SO(3)}$ be a unit variance Gaussian field such that $g^f=g_{\xi,s}$. Then, the formulas in \cref{thm:main1} for $\E\{\LK_i(f\ge u)\}$ hold.
\end{corollary}
\subsubsection{Spin bundles and spin fields}\label{sec:spinbun}
The complex spin field $X$ at \cref{eq:f} can also be seen as a Gaussian section $\sigma_X$ of a complex line bundle $\mathcal{T}^{\otimes s}\to S^2$ over the two-sphere, named the \emph{spin--$s$ bundle}; see \cite{libro,geller2008spin,GeoSpin2022,stecconi2021isotropic}. The notation $\mathcal{T}^{\otimes s}$ was introduced in \cite{GeoSpin2022,stecconi2021isotropic}, which is motivated by the fact that $\mathcal{T}^{\otimes s}$ is the $s^{th}$ tensor power of the complexified tangent bundle $\mathcal{T}^{\otimes 1}=TS^2$. From this point of view, the left-invariance (in law) of $f$ and $X$ translates into the invariance (in law) of $\sigma_X$ under the natural action of $SO(3)$ on the bundle $\mathcal{T}^{\otimes s}$; see \cite{GeoSpin2022}. 

The passage from $\sigma_X$ to $X$ can be roughly explained as follows: any realization of $\sigma_X$ is a section of $\mathcal{T}^{\otimes s}$ \cite{GeoSpin2022,stecconi2021isotropic,articolo_dei_fisici_published}, over the two-sphere. As such, it can be viewed as a function $X_\psi(\f,\theta)$ of the polar coordinates $\theta, \f$ of a point $x\in S^2$, depending on an additional angle $\psi$, which represents a reference tangent direction. Interpreting $\psi, \f,\theta$ as the Euler angles --- coordinates on $SO(3)$ (see \cref{def:eulerangles} below) --- one obtains a function $X$ on $SO(3)$ by setting $X(\psi, \f,\theta):=X_\psi(\f,\theta)$. The function $X$ so constructed satisfies \cref{eq:Xspins}.

From the point of view of spin functions on the sphere, the natural decomposition is with respect to the \emph{spin-weighted spherical harmonics} $Y_{m,s}^l=\sqrt{{(2l+1)}{(4\pi)^{-1}}}D^l_{m,s}$\footnote{The constant comes from the normalization: $\|Y^l_{m,s}\|_{L^2(S^2)}=1$.}, so that 
\begin{equation} 
\sigma_X= \sqrt{\frac{4\pi}{2l+1}}\sum_{l=|s|}^{\infty}c_l\sum_{m=-l}^l \gamma^l_{m,s} Y^l_{m,s}.
\end{equation}
See \cite[Remark 3.6]{stecconi2021isotropic}, \cite[Section 5.2.2]{malyabook}.
\subsubsection{Extension to non-integer spin on $SU(2)$}\label{sec:SU(2)}
The model considered in \cref{eq:f} does not include all spin random fields. Indeed, we assume the spin weight to be an integer, whereas in general it can take values in $ \frac{1}{2}\Z$. The set $\{\mathcal{T}^{\otimes s}\colon s\in \frac{1}{2}\Z\}$ is the complete list of the complex line bundles over $S^2$, up to isomorphism. To include non-integer spin fields, the correct framework is that of random fields on $SU(2)$; see \cite{stecconi2021isotropic}.
This space is diffeomorphic to $S^3\cong SU(2)$ and there is a Riemannian double covering $(SU(2),g_{2S^3})\to (SO(3),g_{\R^6})$, if $SU(2)$ is given the metric $g_{2S^3}$ of a round $3$-sphere of radius $2$; see \cite[Proposition 2.3]{stecconi2021isotropic}. From \cref{thm:main2}, we can deduce that the formulas found in \cref{thm:main1} remain true for non-integer spin.
\begin{corollary}\label{cor:xiesSU(2)}
Let $s\in \frac{1}{2}\Z$. Let $f=\{f(p)\colon p\in SU(2)\}$ be a Gaussian field defined as in \cref{eq:f}. Let us consider the Riemannian metric $g:=g_{2S^3}$ on $SU(2)$, then the formulas in \cref{thm:main1} compute $\frac{1}{2}\E\{\LK_i^g(f\ge u)\}$. More in general, the same holds for any $f$ for which $g^f$ has eigenvalues $\xi^2,\xi^2,s^2$ with respect to $g$. 
\end{corollary}
\begin{proof}
Since $SU(2)\to SO(3)$ is a Riemannian double covering, all the local Riemannian quantities are the same, i.e., the integrands in the formulas of \cref{thm:main2}, applied to $f$, are constant computed with the same formulas used to compute their analogues in \cref{thm:main2}. Indeed, the formulas for $E_1(\xi^2,\xi^2,s^2)$ and $E_2(\xi^2,\xi^2,s^2)$ of \cref{thm:E1E2} hold for all $s>0$. The factor $\frac{1}{2}$ is due to the equation $\vol (SU(2))=2 \ \vol (SO(3))$.
\end{proof}
\begin{remark}\label{rem:berger}
Let $\hat{g}_{\xi,s}$ be the pull-back of $g_{\xi,s}$ to $SU(2)$ and let us see the manifold $SU(2)$ as the three-sphere $S^3$. Then, the standard round metric on $S^3$ is $\frac{1}{4}\hat g_{1,1}$. From \cref{pr:leftrightinvariance}, it follows that the class of Riemannian metrics $\{ \frac{1}{4}\hat g_{1,t}:t>0\}$  coincides with the class of \emph{Berger metrics},  see \cite{berger,urakawa,Gadea_Oubina_2005}.
\end{remark}
\subsubsection{Riemannian waves}
The above corollary can be applied to the case of Riemannian random waves, defined as in \cite{zelditch_2009_rczRRW, Canzani2020-ty}, on $SO(3)$ and $S^3$. These two ensembles are essentially the same and correspond to random spherical harmonics on $S^3$ {(see \cite{kuwabara}, or \cite[Proposition 3.5]{stecconi2021isotropic})}. Let $f_{s}^l$ be independent, for $s=-l,\dots,l$, each defined as in \cref{eq:f}, with the only one non-zero coefficient $c_l=\sqrt{2}$. We say that $f^l_s$ is \emph{monochromatic} of spin $s$. Then, for any $l\in \N$, the field
\begin{equation} \label{eq:rw}
f^l=\frac{1}{\sqrt{2l+1}}\left( f^l_{-l}+\dots +f^l_l \right)
\end{equation} 
is the Riemannian wave of $SO(3)$ with eigenvalue $\lambda_l^{SO(3)}=-l(l+1)$, normalized so as to have unit variance. This is a consequence of \cite[Proposition 3.5]{stecconi2021isotropic}. Moreover, these fields must be left-right-invariant, hence, their Adler-Taylor metric satisfies \cref{eq:conformal} with $g^f=\frac{-\lambda_l}3g_{1,1}$.\footnote{Let $f$ be a unit variance random Laplace-Beltrami eigenfunction on a Riemannian manifold $(M,g)$, so that $\Delta f=\lambda f$ for some $\lambda\le 0$. In case $f$ is homothetic, namely if there exists $\xi$ such that $g^f=\xi^2g$, then the constant $\xi$ is given: $g^f=\frac{-\lambda}{\dim M}g$ because of Green's formula:
\begin{equation} 
-\lambda \vol(M)=-\E \int_M f\Delta f=\E \int_{M}\|df\|_g^2=\dim M\xi^2\vol(M).
\end{equation}} 
It follows that  \cref{cor:xies} applies where we set 
\begin{equation}\label{eq:xisiset} 
\xi^2:=s^2:=\frac{l(l+1)}{3}.
\end{equation} 
The same $f_l$ can be interpreted as a field on $SU(2)$ for all $l\in \frac12 \N$, following the discussion at \cref{sec:SU(2)}, for which  \cref{cor:xiesSU(2)} can be applied with $\xi,s$ as above. From the isometry $SU(2)\cong 2S^3$, we deduce that $f_l$ is also a random hyper-spherical harmonic on $S^3$ with eigenvalue $\lambda_l^{S^3}=-2l(2l+2)$.\footnote{Since $\Delta_{S^3}=4\Delta_{2S^3}$. Note that, each realization of $f_l$ is the restriction to $S^3$ of an harmonic polynomial on $\R^4$ of degree $2\ell$.}
Thus the formulas in \cref{thm:main1}, again with the same $\xi,s$ as above, compute also
\begin{equation} 
2^{i-1}\E\{\LK_i^{S^3}(f_l\ge u)\}.
\end{equation}
Here, the fields $f_l$ (either on $SO(3), SU(2)$ or $S^3$) are all homothetic; hence, the latter formulas could also be proven with \cite[Theorem 13.4.1]{AdlerTaylor}. 

Let us consider the Berger sphere $S^3_t:=(S^3,\frac{1}{4}\hat{g}_{1,t})$, see \cref{rem:berger}. This family of metrics is the canonical variation of the round metric on $S^3=S^3_1$, in the sense of \cite[Section 5]{berg_bourguignon}. Therefore, by \cite[Proposition 5.3]{berg_bourguignon}, the monochromatic spin $s$ fields of type $f^l_s$ are random eigenfunctions of the Laplace-Beltrami operator $\Delta_t$ of $S^3_t$, relative to the eigenvalues 
\begin{equation} 
\lambda_t(l,s)=-4\left( l(l+1)-(1-t^{-2})s^2\right)
\end{equation} 
that can be deduced combining \cite[Corollary 5.5]{berg_bourguignon} with \cite[Corollary 3.8]{stecconi2021isotropic} (and recalling that $\Delta_{SU(2)}=4\Delta_{S^3}$). It follows that the Riemannian random wave of frequency $\lambda$ of the manifold $S^3_t$ is the field
\begin{equation}
f_\lambda^{S^3_t}=\frac{1}{\sqrt{N_t(\lambda)}}\sum_{\{(l,s)\in \Z^2\colon l\ge |s|,\ \lambda_t(l,s)\in [\lambda, \lambda +1)\}}f^l_s,
\end{equation}
where $N_t(\lambda)$ is the number of terms in the above sum. When $N_t(\lambda)=2$ is minimal, then $f_\lambda^{S^3_t}=2^{-\frac12}(f^l_{\bar s}+f^l_{-\bar s})$ for some $l$ and $\bar s$, having the same Adler-Taylor metric as $f^l_{\bar s}$. Comparing the latter with $\frac14 \hat g_{1,t}$, one can see that the formulas in \cref{thm:main1}, with $(\xi,s)=(2(l(l+1)-\bar{s}^2),4\frac{\bar{s}^2}{t^2})$ as above, compute also the quantity
\begin{equation}\label{eq:exberger} 
2^{i-1}\E\left\{\LK_i^{S^3_t}(f_\lambda^{S^3_t}\ge u)\right\}.
\end{equation}
In general, \cref{thm:main2} can be used to compute explicitly the expectation \eqref{eq:exberger} with the method explained in \cref{sec:nonspin}, for independent sums.

\section{Preliminaries}\label{sec:prel}

In this section we give some needed preliminaries on differential geometry, spin random fields and integral and stochastic geometry, needed in the following.
\subsection{The geometry of  \texorpdfstring{$SO(3)$}{SO(3)}}\label{sec:geomSO3}
We consider 
 \begin{equation} 
SO(3)=\{P\in \R^{3\times 3}\colon P^T=P^{-1}, \det P=1\}.
\end{equation}
endowed with the Riemannian metric, which we will denote by $g:=g_{\R^6}$, induced by the inclusion in $ \R^{3\times 3}$, where the latter is endowed with the scalar product $\langle A,B\rangle=\frac{\mathrm{tr}(A^TB)}{2}$.  

With this choice, the map $\pi:P\mapsto P\cdot e_3$, that selects the third column of the matrix $P$, is a smooth Riemannian submersion from $SO(3)$ to the standard round sphere $S^2$, whose fibers are circles of length $2\pi$, see \cite[Proposition 2.3]{stecconi2021isotropic} for details. We will denote this map by $\pi\colon SO(3)\to S^2$.
With the map $\pi$ as projection, $SO(3)$ becomes a circle bundle over the sphere. In fact, one can see that it is isomorphic to the unit tangent bundle the two-sphere 
\begin{equation} 
T^1 S^2:=\{(v,x)\in \R^6\colon  x\in S^2,v\in T_pS^2\}
\end{equation} 
via the map $P\mapsto (P\cdot e_2, P\cdot e_3)$. The metric $g$ on $SO(3)$ corresponds to the one obtained from the identification of $SO(3)$ with the subset $T^1S^2$ of $\R^6$.
We use the parametrization of $SO(3)$ given by Euler angles, following the notations and conventions of \cite{libro,GeoSpin2022,stecconi2021isotropic}.
\begin{definition}\label{def:eulerangles}
For any $\f, \theta, \psi\in \R$, we define $R(\f,\theta,\psi):= R_3(\f)R_2(\theta)R_3(\psi)$, where
\begin{align} 
R_3(\f):=\begin{pmatrix}
\cos \f & -\sin \f & 0 \\
\sin \f & \cos \f & 0 \\
0 & 0 & 1
\end{pmatrix}, \quad R_2(\theta):=\begin{pmatrix}
\cos \theta & 0 & \sin \theta \\
0 & 1 & 0 \\
-\sin \theta & 0 & \cos \theta \\
\end{pmatrix}
\end{align}
\end{definition}
By \cite[Proposition 3.1]{libro}, the restriction of $R$ to the domain $(-\pi,\pi)\times (0,\pi)\times (-\pi,\pi)$ is (the inverse of) a smooth chart of $SO(3)$, whose domain is a full measure subset. 
\begin{lemma}\label{lem:geometrySO3}
The metric $g$ is left and right invariant. The volume of $SO(3)$ is $8\pi^2$. The scalar curvature is $\scal(SO(3))=\frac32$. In the coordinate chart given by the Euler angles $R:(\f,\theta,\psi)\mapsto R_3(\f)R_2(\theta)R_3(\psi)$, the matrix of $g$ is
 \begin{equation} 
\Sigma_{(1,1)}(\h)=\begin{pmatrix}
             1 & \zero & \cos(\h) \\
             \zero & 1 & \zero \\
            \cos(\h) & \zero & 1
         \end{pmatrix}.
 \end{equation}
\end{lemma}
\begin{proof}
Left and right invariance follow from the identity  $$\tr((LAR)^T(LBR))=\tr(A^TB),$$ valid for any pair of matrices $L,R\in SO(3)$.
By \cite[Proposition 2.3]{stecconi2021isotropic}, we know that the map $\pi\colon SO(3)\to S^2$ is a Riemannian submersion, with fibers being circles of length $2\pi$, hence we deduce that the volume is $2\pi\cdot 4\pi$ by the coarea formula. The same Proposition tells us that, with this metric, there exists a local isometry $2S^3\to SO(3)$, hence the scalar curvature should be the same as that of a round three-sphere of radius $2$, which is $\frac{3}{2}$. The matrix $\Sigma(\theta)$ can be computed easily, for instance, the term $(1,3)$ is
\begin{align} 
\frac 12\tr & \left( \frac{\de R}{\de \f}^T\cdot \frac{\de R}{\de \psi}\right) = \\
& =
\frac{d}{dt}\Big|_{t=\f}\frac{d}{ds}\Big|_{s=\psi}
\frac 12\tr\left( 
{R}_3(s) R_3(-\psi)R_2(-\theta) R_3(-t)
R_3(\f)R_2(\theta)
\right)
\\
&=
-\frac 12\tr\left( 
\dot{R}_3(0)R_2(-\theta)\dot R_3(0)R_2(\theta)
\right)
\\
&=
-\frac 12\tr\left( 
\begin{pmatrix}
0 & -1 & 0 \\
\cos \theta & 0 & -\sin \theta \\
0 & 0 & 0 \\
\end{pmatrix}
\begin{pmatrix}
0 & -1 & 0 \\
\cos \theta & 0 & \sin \theta \\
0 & 0 & 0 \\
\end{pmatrix}
\right)
\\
&=\cos \theta.
\end{align} 
\end{proof}
\subsection{Facts on Spin Random Fields}\label{subsec:spinrfs}
The Wigner $D$-functions $D^l_{m,s}\colon SO(3)\to \C$, defined for all $l\in \N$ and $m,s\in \{-l,\dots,l\}$, are the matrix coefficients of the irreducible representations of $SO(3)$. We refer to \cite{libro} (or \cite{stecconi2021isotropic}) for their construction, see also \cite[section 5.2.2]{malyabook} and \cite{GeoSpin2022}. 

In general, actually the function $D^l_{m,s}$ is defined on $SU(2)$, with $l,m,s$ being allowed to be half integers in $\frac{1}{2}\Z$. In our case, with $l,m,s\in \Z$, one has that $D^l_{m,s}(q)=D^l_{m,s}(-q)$ for all $q\in SU(2)$, hence $D^l_{m,s}$ descends to a function of $SO(3)=SU(2)/\Z_2$. 

Their first crucial property is that the matrix $D^l=(D^l_{m,s})_{-l\le m,s\le l}$ defines an irreducible unitary representation $D^l\colon SO(3)\to U(2l+1)$. A second property 
is that 
\begin{equation}\label{eq:leftrightspin} 
D^l_{m,s}\left( R_3(\f)pR_3(\psi)\right)=e^{-im\f}D^l_{m,s}\left( p\right) e^{-is\psi}
\end{equation}
These two properties characterize the functions $D^l_{m,s}$ up to a sign on non-diagonal terms.
A third property is that the collection of all the function $D^l_{m,s}$ form a complete orthogonal system in $L^2(SO(3))$, with 
\begin{equation} 
\|D^l_{m,s}(p)\|^2_{L^2(SO(3))}=\int_{SO(3)}|D^l_{m,s}(p)|^2\vol_{g}(dp)=\frac{\vol_g(SO(3))}{2l+1}=\frac{8\pi^2}{2l+1}
\end{equation}
because of Schur's relations and because of \cref{lem:geometrySO3}. 

An immediate consequence of \cref{eq:leftrightspin} is that the field $X$ defined in \cref{eq:f} satisfies the following almost sure identity
for any $\psi\in \R$
\begin{equation}\label{eq:rspin} 
X(pR_3(\psi))=X(p)e^{-is\psi},
\end{equation}
which is why we say that $X$ has \emph{spin weight} $s$, in accordance with \cite{libro, malya11, NP66, geller2008spin,stecconi2021isotropic, GeoSpin2022, articolo_dei_fisici_published}.
\begin{remark}
The \emph{spin weight} is usually associated to the corresponding spin function, i.e. the random section $\sigma_X$ of the spin bundle $\mathcal{T}^{\otimes s}$ over the sphere. Under the correspondence described in \cref{sec:spinbun}, the pull-back field $X$ satisfies \cref{eq:rspin}, as proved in \cite{BR13}. In \cite[Remark 3.7]{GeoSpin2022} and \cite[section 2.6]{stecconi2021isotropic} this is recalled with the same notation of this paper; there, a function on $SO(3)$ (or $SU(2)$) satisfying \cref{eq:rspin} is said to have \emph{right spin $=-s$}. The reader should be aware, that in some references, the spin weight has the opposite sign. Here, with \cref{eq:rspin}, we are choosing the same convention adopted in \cite{stecconi2021isotropic}, that $\mathcal{T}^{\otimes 1}=TS^2$ is the tangent bundle of the two-sphere, with the standard orientation (the \emph{polar bear} point of view, in the language of \cite[Remark 2.2]{stecconi2021isotropic}), i.e. a function with spin weight $1$ is a vector field on $S^2$.
\end{remark}
The property \eqref{eq:leftrightspin} entails that the function $D^l_{ms}$, once $m,s\in\Z$ are fixed, is essentially determined by its dependence on the Euler angle $\theta$. The function
\begin{equation} 
d^l_{ms}(\theta):=D^l_{ms}\left( R_2(\theta)\right).
\end{equation}
is called \emph{Wigner d-function}. We will be interested mostly in the diagonal ones:
\begin{equation}\label{eq:wignerd}
\begin{aligned}
        d^l_{s,s}(\theta)&=
\sum_{j=0}^{l-|s|} 
(-1)^j\binom{l+s}{j}\binom{l-s}{j} \left(\cos \frac{\theta}{2}\right)^{2(l-j)}\left(\sin \frac{\theta}{2}\right)^{2j}
\\&
=1-\frac{(l(l+1)-s^2)}{2}\frac{\h^2}2+O(\h^4).
\end{aligned}
\end{equation}
The law of the field $X$ is determined by the \emph{circular covariance function} $k\colon \R\to \R$ defined as 
\begin{equation}\label{eq:k} 
k(\theta):=\E\left\{ X(\mathbbm{1})\overline{X(R_2(\theta))}\right\}=\sum_{l=|s|}^\infty c_l^2 d^l_{s,s}(\theta),
\end{equation}
in the way described in \cite[Section 4.2]{GeoSpin2022}. The same can be said for the field $f=\Re X$, that we aim to study.
\begin{lemma}\label{lem:covariance}
The field $f=\Re X$ defined in \cref{eq:f} is left-invariant in law (in the sense of \cite{libro}) and there is a smooth function $K\colon SO(3)\to \R$ such that
\begin{equation}
\E\left\{ f(p)f(q)\right\}=
\G(p^{-1}q)=\frac12\sum_{l\ge |s|}^\infty c_l^2\Re D^l_{s,s}(p^{-1}q).
\end{equation}
Moreover, if $k$ is the circular covariance function of $X$, defined as in \cref{eq:k}, then
\begin{equation}\label{eq:onlything} 
K(R(\f,\theta,\psi))=\frac12 \cos\left({s(\f+\psi)}\right) k(\theta).
\end{equation}
\end{lemma}
\begin{proof}
It is enough to prove it in the case when one $c_l$ is $1$ and all the other are zero. Let $p,q\in SO(3)$. Writing $f=\frac{1}{2}(X+\overline X)$, we obtain 
\begin{align}
\E[f(p)f(q)]&=\frac{1}{2} \Re(\E[X(p)\overline{X(q)}])
=
\frac12 \Re \sum_{m}D^l_{m,s}(p)\overline{D^l_{m,s}(q)}
= \frac12 \Re D^l_{s,s}(p^{-1}q),
\end{align}
where in the last step we used that that matrix $D^l(p)$ is unitary for any $p\in SO(3)$ and that $D^l\colon SO(3)\to U(2l+1)$ is a group homomorphism. The first identity implies that the covariance function of $f$ is left-invariant, thus the Gaussian field $f$ is left-invariant as well. The second identity follows from \cref{eq:leftrightspin}. 
\end{proof}
It proves to be convenient to express everything in terms of the circular covariance function $k$, since this determines completely the law of the Gaussian field $f$, through \cref{eq:onlything}.

In the statement of \cref{thm:main1} and \cref{thm:main2} we take a field $f$ of unit variance. This is to the normalize the constants $c_l$. An immediate consequence of \cref{lem:covariance} is that such normalization is $k(0)=2$ and corresponds to the first set of identities of \cref{eq:intro_norm}. Note that the Cauchy-Schwartz inequality implies that $k$ has a maximum at $0$, hence $k'(0)=0$. In fact, $k$ is an even function, so all of its derivatives of odd order vanish. 
\begin{remark}
With \cref{thm:main1}, we will prove that in the end the value of $\E\left\{ \LK_i(f\ge u)\right\}$ depend only on $s,u$ and on the second order Taylor expansion of $k$ at $0$:
\begin{equation} 
k(\theta)=2-\xi^2\theta^2+O(\theta^4),
\end{equation}
where $\xi^2=-\frac{1}{2}k''(0)$ satisfies the second identities of \cref{eq:intro_norm}.
\end{remark}
\subsection{Facts in Differential Geometry}\label{subsec:diffGeo}
In this section we recall some definitions of objects from tensor calculus and Riemannian geometry and establish our notation for them. We follow closely the setting of \cite[Sections 7.2, 7.3 and 7.5]{AdlerTaylor}, see also \cite{leeriemann}.

For this section, we let $(A,g)$ denote a Riemannian manifold of dimension $\dim M=d$, with boundary $\de A$. For any $u,v\in T_xM$, we write $g(u,v)=\langle u,v\rangle_g$ for the scalar product and $\|v\|_g=\sqrt{g(v,v)}$ for the norm associated with the metric $g$. 
\subsubsection{Double forms and trace}\label{sec:doubletrace}
We keep the same notations $\langle\cdot ,\cdot \rangle_g$ and $\|\cdot \|_g$ for the scalar product and the norm induced by $g$ on the tensor spaces of \emph{double forms}
\begin{equation} 
\Lambda^{k,h}_xT^*_xA:=\Lambda^k T^*_xA\otimes \Lambda^h T^*_xA,
\end{equation}
for any $x\in A$.
The elements $\gamma$ of $\Lambda^{k,h}_xT^*_xA$ are the multilinear functions $\gamma\colon (T_xA)^{k+h}\to \R$ that are skew-symmetric with respect to the first $k$ variables and with respect to the last $h$.  
Given two multilinear functions $m^i : (T_xA)^{k_i}\to \R$, their tensor product $m^{1}\otimes m^{2} 
$ is the multilinear function $(T_xA)^{k_1+k_2}\mapsto \R$ such that $$m^{1}\otimes m^{2}(v_1,\dots,v_{k_1+k_2})=m^1(v_1,\dots,v_{k_1})m^2(v_{k_1+1},\dots,v_{k_1+k_2}).$$
Let $\mathscr{S}_k$ be the set of permutations of $k$ elements and let 
\begin{equation} 
\m{I}_k^m:=\left\{ I=(i_1,\dots,i_k)\in \N^k: 1\le i_1<\dots <i_k\le m\right\}.
\end{equation}
Note that $\m{I}_k^m$ is in bijection with the set of subsets of $\{1,\dots,d\}$ of cardinality $k$, hence the cardinality of $\m{I}_k^m$ is $\binom{d}{k}$.
Given an orthonormal basis $e_1,\dots, e_d$ of $T_xM$ and its dual basis $e^1,\dots,e^d$ of $T^*_xM$, we define
\begin{equation}\label{eq:eI} 
e^I:=e^{i_1}\wedge \dots \wedge e^{i_k}=\sum_{\sigma \in \mathscr{S}_k}\mathrm{sgn}(\sigma) e^{i_{\sigma(1)}}\otimes \dots \otimes e^{i_{\sigma(k)}},
\end{equation}
so that $e^I\in \Lambda^kT^*_xA$ is the multilinear function such that $e_I(v_1,\dots,v_k)=\det (e^{i_a}(v_b))_{1\le a,b \le k}$
 for any $v_1,\dots, v_k\in T_xA$.
Then, any double form $\gamma\in \Lambda^{k,h}T_x^*A$ can be written in a unique way as 
\begin{equation} 
\gamma=
\sum_{I\in \m{I}^d_k,J\in \m{I}^d_h} \gamma_{I,J} e^I\otimes e^J, \quad \text{where}\quad \gamma_{I,J}=\gamma(e_{i_1},\dots, e_{i_k},e_{j_1},\dots,e_{j_h}).
\end{equation}
In other words, the set 
$ 
 \left\{ e^I\otimes e^J \ :\  I\in \m{I}^d_k, J\in \m{I}^d_h
 \right\}
$
is an orthonormal basis of the vector space $\Lambda^{k,h}T^*_xA$. 
Analogously, the tensors $e_I$ constructed as in \cref{eq:eI}, but with indices down, form an orthonormal basis for the space $\Lambda^k T_xA$, which is dual to the basis $e_I$ under the canonical identification $(\Lambda^k T_xA)^*=\Lambda^k T_x^*A$. 

\subsubsection{Trace}\label{eq:trace}
The
metric allows to make the identification $\Lambda^{k,k}T_x^*A\cong \mathrm{End}(\Lambda^k T_xA)$, via the linear map $e^I\otimes e^J\mapsto e_I\otimes e^J$.
The \emph{Trace} of a double form $\gamma \in \Lambda^{k,k}T^*_xA$ is the trace of the corresponding endomorphism, i.e., 
\begin{equation} 
\Tr(\gamma)= \sum_{I\in \m I^d_k}\gamma_{I,I}.
\end{equation}
in particular, we have that $\Tr(e^I\otimes e^J)=\langle e^I,e^J\rangle_g\in \{0,1\}$ as in \cite[section 7.2]{AdlerTaylor}, an equivalent formula is \cite[(7.2.6)]{AdlerTaylor}. If $V\subset T_xA$ is a vector subspace of dimension $j$, we can assume that $V=\ker(e_{j+1})\cap\dots\cap\ker(e_d)$ and we use the symbol
\begin{equation} 
\Tr^{V}\left( \gamma \right) := \Tr(\gamma|_V)= \sum_{I\in \m I^j_k}\gamma_{I,I}
\end{equation}
to denote the trace of the restriction of $\gamma$ to $V^{2k}$. 
Note that the trace $\Tr$ depends on the metric $g$, but note on the choice of orthonormal basis.
When $\gamma \in \Lambda^{1,1} T^*_xM$ is just a bilinear form on $T_xA$, we recover the standard notion of trace with respect to the metric $g$, often denoted by $\tr_g(\gamma)=\Tr(\gamma)$.

Given two double forms $\a\otimes \beta\in \Lambda^{k,h}T_x^*A$ and $ \a'\otimes \beta'\in \Lambda^{k',h'}T_x^*A$, their wedge product is the double form $\a\wedge \a'\otimes \beta\wedge \beta'\in \Lambda^{k+k',h+h'}T_x^*A$. By linearity, this definition is extended to any pair of double forms. In particular, with such a product, the vector space $\oplus_{k=0}^d \Lambda^{k,k}T_x^*A$ becomes a commutative algebra. For any elements $R,S$ of the latter and any $m,n\in \N$, we write
$
R^mS^n$ for the product of their powers in this algebra. This explains the meaning of the expression $\Tr(R^mS^n)$ in \cref{eq:LK3dim}.

\subsubsection{The curvature tensors}
Let $\nabla$ denote the levi-Civita connection of $(A,g)$. The \emph{Riemann tensor} (of type $(0,4)$) at $x\in M$ is the multilinear form $R=R_x\colon (T_xA)^4\to \R$ defined as 
\begin{equation} 
R(u,v,w,z)=g\left( \nabla_u \nabla_v w-\nabla_v \nabla_uw-\nabla_{[u,v]}w, z  \right),
\end{equation}
where $u,v,w,z$ are extended to vector fields in a neighborhood of $x$ (see \cite[(7.5.1)]{AdlerTaylor}).

For any $y\in \de A$, let $\nu(y)\in T_yA$ denote the outer normal vector to the boundary. 
The \emph{second fundamental form} of $\de A$ in the direction $\nu$ at $y\in \de A$ is the bilinear form $S_\nu=S_{\nu(x)}\colon (T_x\de A)^2\to \R$ given by
\begin{equation} 
S_\nu(u,v)= -g(v,\nabla_u\nu),
\end{equation}
for any $u,v\in T_x\de A$, see \cite[(7.5.12)]{AdlerTaylor}.
Because of the well known symmetries of the tensors $R$ and $S$, we have that  for any $x\in A$ and $y\in \de A$, they are double forms:
\begin{equation} 
R\in \Lambda^{2,2}T_x^*A, \quad \text{and} \quad S\in \Lambda^{1,1} T_y^*\de A.
\end{equation} 
Let $e_1,\dots,e_{d}$ be an orthonormal basis of $T_xA$ such that $e_d=\nu$ and let $R_{ijkh}$ and $S_{ij}$ be the coefficients of $R$ and $S$ with respect to it. The \emph{scalar curvature} of $(A,g)$ is the function $\scal\colon A\to \R$ defined as 
\begin{equation} \label{eq:scal}
\scal=\sum_{1\le i,j\le d}R_{ijji}=-2\sum_{1\le i<j\le d}R_{ijij}=-2\Tr^{TA}(R)
\end{equation}
The \emph{mean curvature} of $\de A$ in the outer direction $\nu$ is the function $\hout{\de g}\colon \de A\to \R$ defined as 
\begin{equation}\label{eq:houtTr} 
\hout{\de A}=\frac{1}{d-1}\Tr^{T\de A}
(S)
\end{equation}
expressing the average of the principal curvatures, namely the eigenvalues of $S.$

    Let us consider the case when $\de A$ has dimension $d-1=2$. We denote by $R^\de$ and $\scal^{\de}$ the Riemann tensor and scalar curvature of $\de A$ (and not that of $A$). Then, we have that the \emph{Gaussian curvature} $\kappa$ of $\de M$ can be expressed as:
    \begin{multline}\label{eq:gauss} 
\kappa=R_{1221}^\de={\det}_g(S)+R_{1221}=\frac12 \Tr^{T\de A}(S^2)-\Tr^{T\de A} (R)\\
=-\Tr^{T\de A}(R^\de)=\frac12\scal^\de.
    \end{multline}
The first and second of the above identities follow from \emph{Gauss equation} \cite[Theorem 8.5]{leeriemann}, expressing his \emph{Theorema Egregium}\footnote{They imply that, when $A$ is flat, $\kappa$ is the product of the eigenvalues of $S$, i.e., the principal curvatures.}; the third is a straightforward computation following from the definition of the trace of double forms, see \cref{sec:doubletrace}; the last two identities are deduced from \cref{eq:scal}.

\subsubsection{The gradient and the Hessian}
Let $f\in \mC^\infty(M)$ be a smooth function on a Riemannian manifold $(M,g)$. 
For a tangent vector $v\in T_pM$, we write $v(f)=d_pf (v)$ for the differential of $f$ in the direction $v$.
The \emph{gradient} of $f$ at $p\in M$ is the tangent vector $\nabla f(p)$ such that $g(\nabla f(p),v)=d_p f(v)$ for any $v\in T_pM$. The \emph{Hessian} of $f$ at $p$ is the symmetric bilinear form $\hess_p f\colon (T_pM)^2\to \R$ such that 
\begin{equation}\label{eq:hess} 
\hess f(u,v)= g\left(\nabla_u \nabla f, v\right) 
\end{equation}
for any pair of vector fields $u,v$. In other words, $\hess f=\nabla df$.
\subsubsection{Regular Excursion set}
Recall the definition of regular value, from \cite{Hirsch}. 
\begin{definition}
Let $f\in \mC^\infty(M)$ and let $u\in \R$. $u$ is said to be a \emph{regular value} for $f$ if there is no point $p\in f^{-1}(u)$ for which $d_pf=0$.
\end{definition}
The following is a classical consequence of the implicit function theorem, see \cite{Hirsch}. 
\begin{lemma}\label{lem:regval}
Let $f\in \mC^\infty(M)$ and let $u\in \R$ be a regular value. Then, the excursion set $A_u(f):=\left\{ x\in M\colon f\ge u\right\}$ is a smooth manifold with boundary $\de A_u(f)=f^{-1}(u)$. We will say that $A_u(f)$ is a \emph{regular excursion set}.
\end{lemma}

Applying a version (\cite[Proposition 4.11]{MathiStec}) of Bulinskaya's Lemma (see \cite[Proposition 6.12]{AzaisWscheborbook} for a classical statement) we will show, in \cref{lem:kr} below, that in our case we have the following.
\begin{lemma}\label{lem:smoothAu}
If $f$ is a smooth Gaussian random field with unit variance, then for any $u\in \R$, then $A_u(f)$ is almost surely a regular excursion set.
\end{lemma} 
\begin{proof}
 The statement is contained in that of \cref{lem:kr}.
\end{proof}

In the case of a regular excursion set, by a straightforward calculation combining \cref{eq:houtTr} and \cref{eq:hess}, we can write the mean outer curvature of $\de A_u(f)$ as:
\begin{equation}\label{eq:Hout}
\hout{\de A_u} =  
\frac12 \frac{\Tr (\hess f) - \hess f \left(\frac{\nabla f}{\|\nabla f\|},\frac{\nabla f}{\|\nabla f\|}\right) }{\|\nabla f\|}.
\end{equation} 
We will use this formula in the proof of \cref{thm:main2}.
Notice that the expression on the right of \cref{eq:Hout} is a function defined for all $p\in M$ such that $\nabla f(p)\neq 0$. 
\subsection{The Adler-Taylor metric}\label{sec:AT}
Let $M$ be a smooth manifold. Let $f=\{f(p)\colon p\in M\}$ be a smooth Gaussian random field. For us, this means that $f$ is a collection of jointly Gaussian centered random variables $f(p)$ indexed by $p\in M$, defined on some abstract probability space $(\Omega, \mathscr{S},\P)$ and such that with probability $\P=1$, the function $p\mapsto f(p)$ is of class $\mC^\infty(M)$. This definition is equivalent to any other from \cite{AdlerTaylor,AzaisWscheborbook,libro,dtgrf}, and \cite[Appendix A]{NazarovSodin2016ThatOne}. In the following we mostly refer to \cite{AdlerTaylor,AzaisWscheborbook}.
We say that $f$ has \emph{unit variance} if $f(p)\sim \m N(0,1)$ for very $p\in M$. 

We recall the definition of the metric associated to a Gaussian field $f$.
\begin{definition}\label{def:atm}
Given a Gaussian field $f$ on $M$, we define the bilinear form $g^f_p\colon (T_pM)^2\to \R$ as
\begin{equation} 
g^f_p(u,v)=\E\left\{ d_pf(u)\cdot d_pf(v)\right\},
\end{equation}
for any $u,v\in T_xM$. 
If $g^f_p$ is positive definite for every $p\in M$, then we say that $f$ is \emph{non-degenerate} and we call $g^f$ the \emph{Adler-Taylor metric} of $f$.
\end{definition}

In our setting, the Adler-Taylor metric $g^f$ is a smooth Riemannian metric on $M$, which in general have nothing to do with the original one $g$. We will use the notation $X^f$ for any Riemannian object $X$ (e.g. the Riemann tensor $R^f$, the Hessian $\hess^f\f$ of a function $\f$) computed with respect to the metric $g^f$.

In case of a non-degenerate unit variance Gaussian field $f$, the Riemannian metric $g^f$ has the following important properties, which we will use in the proof of \cref{thm:main2}.
\begin{lemma}\label{lem:1}
Let $\{v_1,\ldots, v_d\} \subset T_pM$ be a $g^f$-orthonormal frame, that is, $g^f(v_i,v_j)=\delta_{ij}$, where $\delta$ denotes the Kronecker delta. Then, the random variables $\{df(v_i) : i=1,\ldots,d\}\cup \{f(p)\}$ are iid standard Gaussian. 
\end{lemma}

\begin{proof}
By definition of $df$, for any tangent vector $v_i\in T_pA$, $i=1,\ldots,d$, the mapping $df(v_i)$ defines a real-valued random variable. Since $f$ is centered Gaussian, $df(v_i)$ is centered Gaussian, too, and satisfies \begin{equation}
        \E[df(v_i)df(v_j)] = g^f(v_i,v_j) = \delta_{ij}, \quad i,j=1,\ldots,d.
    \end{equation}
Finally, by differentiating the identity $\E|f(p)|^2=1$, we deduce that $f(p)$ and $d_pf$ are independent.
\end{proof}
\begin{lemma}
Let $\hess_p^ff$ be the Hessian of $f$ at $p$, computed with respect to $g^f$. Then $\hess_p^ff$ and $d_pf$ are independent. 
\end{lemma}
\begin{proof}
See \cite[Eq. (12.2.11)]{AdlerTaylor}.
\end{proof}
The correlation of $\hess^f_pf$ and $f(p)$ is clear too, as described at \cite[Eq. (12.2.12)]{AdlerTaylor}. However, we will need the following slightly more general description, in which the Hessian is taken with respect to a metric $g$ possibly different than $g^f$.
\begin{lemma}\label{lem:2}
Let $g$ be a Riemannian metric on $M$. Let $\hess_pf$ be the Hessian of $f$ at $p$ computed with respect to $g$.
We can write the conditional expectation of $\hess_pf$ given $f$ as 
\begin{equation}
\E\left[ \ \hess_p f \mid f(p)=u \right] = - u g^f,
\end{equation}
where $g^f$ is the Adler-Taylor metric of $f$.
\end{lemma}
\begin{proof}
    By Gaussian regression, we can write \begin{equation}
        \hess_p f = f(p) A   + X ,
    \end{equation}
    where $A,X$ are (random) bilinear symmetric forms on $T_pM$ such that $X$ is uncorrelated from $f(p)$. Let us show that $A = -g^f_p$ and $X = \hess_p f(p) + f(p) g^f_p $; in particular, $X$ is Gaussian and $\E[X]=0$. For any vectors $v,w\in T_pM$, we have\footnote{In the inner computations we don't write the dependence on $p$.}
\bega
\E[f(p) \hess_p f (v,w)] & = \E[ f g(\nabla_v \nabla f, w)] \\
& = \E[g(\nabla_v (f \nabla f) - v(f) \nabla f, w)] \\
& = g(\nabla_v\E[f\nabla f],w) -\E[v(f)g(\nabla f,w)] \\
& = -g_p^f(v,w)
\eega
since $\E[f(p)\nabla_p f] = 0$ at any point $p\in M$, being $f$ a unit variance Gaussian random field. Hence, at any $p\in M$ and for any vectors $v,w\in T_pM$, we have that 
    \begin{equation}
        \E[f(p) X(v,w)] = \E[f  \hess f + f^2  g^f ] = -g^f + \E[f^2] g^f  = 0.
    \end{equation}
    Being Gaussian, $f(p)$ and $X(v,w)$ are independent at any point, for any $v,w$. To conclude, \begin{equation}
        \E[ \ \hess_p f \mid f(p) = u] = - u g_p^f + \E[X] = -u g^f_p.
    \end{equation}
\end{proof}

\subsection{The Lipschitz-Killing curvatures in dimension \texorpdfstring{$3$}{3}}\label{sec:LK3d} 
Let $(A,g)$ be a three-dimensional manifold with boundary $\de A$.
For $i\in \{0,1,2,3\}$, the $i^{th}$ Lipschitz-Killing curvature measure $\mathcal L_i(A,\cdot)$ (see \cite[Def. 10.7.2]{AdlerTaylor}) of $A$ is defined for any Borel subset $B\subset A$ as  
\bega\label{eq:LK3dim}
\mathcal L_i(A, B) &=\sum_{m=0}^{\left\lfloor \frac{2-i}2 \right\rfloor} \frac{(-1)^{m-i} \Gamma\left(\frac{3-i-2m}2\right)}{ m!(2-i-2m)! 2^{1+m}} \pi^{-(3-i)/2} \times \\
& \times \int_{\partial A \cap B} \Tr^{T\de A}\left( R^m S^{2-i-2m}_{\nu(t)}\right)\mathcal H^2(dt) \\
& \quad + 
\sum_{m=0}^{\left\lfloor \frac{3-i}2 \right\rfloor} \frac{(-1)^m (2\pi)^{-(3-i)/2}}{m!(3-i-2m)!}  \int_{A^{\circ} \cap B}  \Tr\left(R(m,i)\right)\mathcal H^3(dt),
\eega
where \begin{itemize}
\item $\mathcal H^2(dt)$ and $\mathcal H^3(dt)$ denote, respectively, the Riemannian volume measure on $\partial A$ and on $A^\circ$ (see also \cref{subsec:setting}) induced by the metric $g$;
\item $\nu(t)$ denotes the outer normal unit vector at the point $t\in\partial A$; 
\item $R^m$ and $S^l_{\nu}$ denote, respectively, the $m^{th}$ power of the Riemann tensor and the $l^{th}$ power of the second fundamental form at the vector $\nu$, and $R(m,i)$ is defined as follows \begin{equation}\label{eq:defRmi}
R(m,i)=\begin{cases}
0 & \text{if } i=0,2 \text{ or } (m,i)=(0,1) \\
R & \text{if } (m,i)=(1,1) \\
1 & \text{if } (m,i)=(0,3)
\end{cases}
\end{equation}
\item $\Tr$ denotes the trace of double forms as in  \cref{sec:doubletrace});
\item $\Gamma(x)= \int_0^\infty t^{x-1} e^{-t} dt$ denotes the Gamma function.
\item We adopt the convention that $\sum_{m=0}^{\left\lfloor -\frac12 \right\rfloor}=0$.
\end{itemize}
The $i^{th}$ \emph{Lipschitz-Killing curvature}, or \emph{intrinsic volume}, of $A$ is defined as the real number $\mathcal{L}_i(A):=\mathcal{L}_i(A,A)$ cf. \cite[Equation (10.7.3)]{AdlerTaylor}. 

\begin{remark}
    For the interested reader, we report in the Appendix the details why \cite[Definition 10.7.2]{AdlerTaylor} breaks down to the previous expression for a $3$-dimensional manifold with boundary.
\end{remark}
\begin{remark}\label{rem:tube}
The Lipschitz-Killing curvatures of $A$  can be equivalently characterized though Weyl's tube formula \cite[Theorem 10.5.6]{AdlerTaylor} as follows. Let $\iota\colon A\to \R^n$ be an isometric embedding of $A$ in $\R^n$. Let $\m H^n$ denote the Lebesgue measure of $\R^n$ and let $\e\B^n$ be the ball of radius $\e>0$ in $\R^n$. Then, the $\e$--\emph{tube} around $\iota(A)$ has volume
\begin{multline}\label{eq:tube} 
\m H^n\left( \iota(A)+\e\B^n
\right)= \LK_3(A)\e^{n-3}\w_{n-3}
+\LK_2(A)\e^{n-2}\w_{n-2}\\+\LK_1(A)\e^{n-1}\w_{n-1}+\LK_0(A)\e^n\w_n,
\end{multline}
for any small enough $\e>0$.
\end{remark}
Notice that the curvature terms appearing in \cref{eq:LK3dim} are 
\begin{equation} 
\Tr^{T\de M}\left( R \right), \Tr^{T\de M}\left( S_\nu \right), \Tr^{T\de M}\left( S_\nu^2 \right), \Tr(R).
\end{equation} 
In particular, the formula for $i=0$ coincides with the Gauss--Bonnet theorem (see \cref{rem:Gaussbonnet}) in the form of \cref{eq:Gaussbonnet}.

\begin{remark}\label{rem:Gaussbonnet}
When $A$ is a three-dimensional manifold with boundary, the Gauss Bonnet Theorem \cite[Theorem 9.3]{leeriemann} applied to $\de A$, together with the additivity of the  Euler characteristic $\chi$\footnote{Let $\tilde A$ be the union of two copies of $A$ glued along the boundary. Then $A$ is a closed odd dimensional manifold and its Euler characteristic is $0=2\chi(A)-\chi(\de A)$. The Gauss-Bonnet theorem applied to $\chi(\de A)$, then gives \cref{eq:Gaussbonnet}.} implies that
\begin{align}\label{eq:Gaussbonnet} 
\LK_0(A)&=\chi(A)=\frac{1}{4\pi}\int_{\de A} \kappa \ d\m H^2_g \\
&=
\frac{1}{8\pi}\int_{\de A} \Tr^{T\de A}(S^2)\ d\m H^2_g-\frac{1}{4\pi}\int_{\de A}\Tr^{T\de A} (R)  \ d\m H^2_g,
\end{align}
where we used \cref{eq:gauss} in the last identity. This formula coincides with the case $i=0$ of \cref{eq:LK3dim}.
\end{remark}

\begin{proposition}\label{pr:ABLK1} 
Let $\hout{\de A}$ denote the outer mean curvature of the boundary and $\scal$ the scalar curvature of the Riemannian manifold $(A,g)$, defined as in \cref{eq:Hout} and \cref{eq:scal}, respectively. The first Lipschitz-Killing curvature measure of $A$ is
\begin{equation}\label{eq:LK1AB}
        \mathcal L_1(A,B) = -\frac{1}{\pi} \int_{\de A\cap B} \mathrm H_{\de A}^{\out} (t) \mathcal H^2(dt)  + \frac{1}{4\pi} \int_B \scal(t) \mathcal H^3(dt),
\end{equation}
for any Borel subset $B\subset A$.
\end{proposition}
\begin{proof} Let us consider the three terms arising in \cref{eq:LK3dim}, when $i=1$, since the first sum has the only addendum corresponding to $m=0$ and the second sum depends on $m=0,1$. By definition of the mean outer curvature at a point $t\in\partial A$, and computing the constants, we obtain the first term of \cref{eq:LK1AB}. The second term in zero because we are integrating $R(0,1)=0$. The last term involves the integral of the trace of $R(1,1)=R$, the Riemann tensor form of $A$. Since its trace satisfies  $\Tr(R)=-\frac{1}{2}\scal$, we obtain the second term in \cref{eq:LK1AB}.
\end{proof}  
\begin{proposition}\label{pr:ABLK2}
The second and third Lipschitz-Killing curvature measures of $A$ are
\begin{equation}
\mathcal L_2(A,B) = \frac{1}{2} \mathcal H^2(\de A \cap B), \quad \text{and} 
\mathcal L_3(A,B) = \vol(A \cap B),
\end{equation}
for any Borel subset $B\subset A$.
\end{proposition}
\begin{proof} 
The case $i=3$ is clear. Let us prove the case $i=2$.
From \cref{eq:LK3dim}, we have to compute two terms. The second one vanishes as the integral involves $R(0,2)=0$. Integrating $R(0,3)=1$, computing the costant, which is $\frac12$, we conclude.
\end{proof}
\begin{example}
To double-check the constants in the above formulas, we test them in two special cases: the ball and the sphere. When $A=\B^3$ the three-dimensional unit ball, the boundary $\de A=S^2$ is the unit sphere and we have $\kappa=1=-\hout{\de A}$ and $\chi(A)=1$. Because of Weyl's formula (\cref{eq:tube}) we have 
\begin{align}
\vol((1+\e)\B^3)=\w_3+\left( 3\frac{\w_3}{\w_1}\right) \e \w_1+\left( 3\frac{\w_3}{\w_2}\right)\e^2\w_2+ 
\w_3\e^3,
\end{align}
from which we can see that $\LK_3(\B^3)=\frac{4}{3}\pi$, $\LK_2(\B^3)=2\pi$, 
\begin{align} 
\LK_1(\B^3)&=4=-\frac{1}{\pi}\int_{\de \B^2}(-1)d\m H^2, \\
\LK_0(\B^3)&=1=\chi(\B^3)=\frac{1}{4\pi}\int_{\de \B^3} 1 d\m H^2.
\end{align}
Similarly, for $A=S^3$ the three-dimensional unit sphere, we have $\chi(A)=0$, $\scal=6$ and 
\begin{align}
\vol((1+\e)\B^4)-\vol((1-\e)\B^4)=\left( 8\frac{\w_4}{\w_1}\right) \e \w_1+ 
\left( 8\frac{\w_4}{\w_3}\right) \e^3\w_3,
\end{align}
from which we can see, that $\LK_3(S^3)=4\w_4$, $\LK_2(S^3)=0=\LK_0(S^3)$ and 
\begin{equation} 
\LK_1(S^3)=4\w_4 \frac{3}{ 2\pi}=\frac{1}{4\pi}\int_{S^3}6 d\m H^3.
\end{equation}
\end{example}

\section{Proof of the main results}\label{sec:proofs}

This section is devoted to the proof of \cref{thm:main1,thm:main2}.  First, we prove \cref{thm:main2}. Then, we use the latter to derive \cref{thm:main1}.

The main technical tool used in the next proof is the Kac-Rice formula, which allows to compute the expectation of an integral along the level set of a suitably regular Gaussian field. We refer to \cite[Theorem 6.10]{AzaisWscheborbook} for a classical statement of the formula. More precisely, we will apply the \emph{$\alpha$-formula} of \cite[Theorem 6.2]{MathiStec}, because it is a stronger statement that can be applied directly in our manifold setting.
A concise version of it, that we will use in our proofs, is the following.

\begin{lemma}[Kac-Rice formula from \cite{MathiStec}]\label{lem:kr}
    Let $(M,g)$ be a smooth manifold of dimension $3$ and let $f=\{ f(p)\colon p\in M\}$ be a smooth Gaussian field over $M$, possibly non-centered and such that for all $t\in M$, $f(t)$ has a density $\rho_{f(t)}\colon \R\to [0,+\infty)$. Then, for any $u\in \R$, the set $A_u(f)=\{t\in M:f(t)\ge u\}$ is almost surely a regular excursion set (defined as in \cref{lem:regval}) and the following formula holds:
    \begin{gather} 
\E\left\{ \int_{f^{-1}(u)}\a(f,t) \m H^{2}(dt)\right\}=
\int_M \E\left\{\|\nabla f (t)\| \a (f,t)\Big| f(t)=u\right\} \rho_{f(t)}(u)\m H^{3}(dt),
    \end{gather}
    for any $\a\colon \mC^1(M,\R)\times M\to \R$ Borel measurable function. Here, $\m H^i$ denotes the $i$-th Hausdorff measure relative to the metric $g$.
\end{lemma}
\begin{proof}
We consider the Gaussian field $X(\cdot)=f(\cdot)-u$. It is smooth and for any $t\in M$ the Gaussian variable $X(t)$ is non-constant. Then, by \cite[Proposition 4.11]{MathiStec}, we deduce that $X$ satisfies the \text{z-KROK} hypotheses, in the sense of \cite[Definiton 4.1]{MathiStec}. In particular, condition (2) of the latter definition yields that $0$ is a regular value of $X$, almost surely; this is equivalent to $u$ being a regular value of $f$, hence, by \cref{lem:regval}, we conclude that $A_u(f)$ is a regular excursion set almost surely. Moreover, from the fact that $f$ is a \text{z-KROK} field, it follows that we can apply \cite[Theorem 6.2]{MathiStec} to $X$, with $k=1$. Thus, we conclude.
\end{proof}

\subsection{Proof of \texorpdfstring{\cref{thm:main2}}{ }}

We divide the proof in four parts, one for each Lipschitz-Killing curvature. Let us recall the definition of $A_u(f)$ as in \cref{eq:introAu}. By \cref{lem:smoothAu}, $A_u$ is a smooth submanifold of $M$ with boundary; hence, \cref{pr:ABLK1,pr:ABLK2} apply a.s. In what follows, we use the notation $A_u$ instead of $A_u(f)$.

\begin{proof}[First part, $\E \mathcal L_0^g(A_u)$]

Being equal to the Euler characteristic, which is a topological invariant, we have that $\mathcal L_0^g = \mathcal L_0^f$. 
The formula follows by a direct application of \cite[Theorem 12.4.1]{AdlerTaylor}, noting that $d\vol^f = \sqrt{abc} \ d \vol_g$ and $\tr(R^f)=-\frac12\scal^f$. \cref{lem:AdlerTaylorSieteDeiCani} below ensures that we are in the position of applying such theorem. Here is why. In the statement of \cite[Theorem 12.4.1]{AdlerTaylor}, it is assumed that $f$ satisfies the Conditions of \cite[Corollary 11.3.5]{AdlerTaylor}; however, inspecting the proof of \cite[Theorem 12.4.1]{AdlerTaylor}, it is clear that only the two conclusions of \cref{lem:AdlerTaylorSieteDeiCani} are used: first, that $f$ is a Morse functions a.s., for representing the Euler characteristic as a sum of critical points; and second, the formula of \cite[Theorem 12.1.1]{AdlerTaylor}, which simplifies to \cite[Corollary 12.1.2]{AdlerTaylor} for Gaussian fields, and whose assumptions correspond to point (1) of \cref{lem:AdlerTaylorSieteDeiCani}.
\end{proof}

\begin{lemma}\label{lem:AdlerTaylorSieteDeiCani}
    Let $f$ be as in \cref{thm:main2}. Then, 
    \begin{enumerate}[(1)]
        \item for any $t\in M$, the Gaussian vector $(f(t),\nabla f(t))$ is non-degenerate.
        \item $f$ is a Morse function, almost surely. 
    \end{enumerate}
\end{lemma}
\begin{proof}
Note that if $g^f$ has everywhere strictly positive eigenvalues, then it is a Riemannian metric. Let us fix $t\in M$.  Since $f$ is unit variance, it follows that $f(t)$ and $\nabla f(t)$ are independent, hence Point (1) is equivalent to the non-degeneracy of $\nabla f(t)$ only, which follows from \cref{lem:1}. Now, Point (1) implies that the Gaussian random vector field $X=\nabla f$ satisfies the hypotheses of \cite[Theorem 3.2]{stecconi2021kacrice}, with $E=TM$ the tangent bundle, and $W$ being the zero section of $TM$, therefore $X$ is transverse to $W$ with probability one. By a standard differential-geometric argument (see \cite[Chapter 6, Section 1]{Hirsch}), a smooth function $f$ is Morse if and only if its gradient $\nabla f$ is transverse to $W$. Hence, Point (2) is proven.
\end{proof}

\begin{proof}[Second part, $\E \mathcal L_3^g(A_u)$]
By assumption, see \cref{eq:intro_norm}, for any $p\in M$ the random variable $f(p)$ is standard Gaussian random variable. Hence, changing the order of integration, we get
\begin{equation}
    \E \mathcal L_3^g(A_u) = \E\left[\vol_g (A_u^{\circ})\right] = \vol_g(M) (1-\Phi(u)) .
\end{equation}
\end{proof}
\begin{proof}[Third part, $\E \mathcal L_1(A_u)$] 
We may apply Proposition~\ref{pr:ABLK1} with $A=B=A_u$. Then, 
\begin{equation}
        \mathcal L_1^g(A_u) = -\frac{1}{\pi} \int_{\de A_u} \mathrm H_{\de A_u}^{\out} (t) dt  +  \frac{1}{4\pi}\int_{A_u} \scal(t) dt \qquad \text{a.s.}
    \end{equation}
    
    By \cref{lem:kr} we have that
    \bega
        \E[\mathcal L_1^g (A_u)] &= -\frac{1}{\pi} \int_{M} \E[ \ \| \nabla f(t)\| \mathrm H_{\de A_u}^{\out} (t)  \mid f(t)=u] \ p_{f(t)}(u) d\vol_g(t)\\ &
        \hspace{5cm}
        + \frac{1}{4\pi} \E\left[ \int_{A_u} \scal(t) d\vol_g(t)\right]   \\
        &=: I_1 + I_2,
    \eega
where: \begin{itemize}
\item $\E[ \ \cdot  \mid f(t)=u]$ denotes the conditional expectation with respect to $\{f(t)=u\}$. Note that the mean outer curvature $\mathrm H_{\de A_u}^{\out}(t)$ of the boundary is well-defined while conditioning on $\{f(t)=u\}$.
\item Both integrals are with respect to the volume form of the metric $g$, denoted by $d \vol_g(t)$.
\item $p_{f(t)}$ denotes the density of the random variable $f(t)$, which is standard Gaussian. Hence, its value does not depend on $t$.
\end{itemize}   
Then, conditioning on $\frac{\nabla f (t)}{\|\nabla f (t)\|} = v$ and applying Lemma~\ref{lem:2}, we deduce that 
\begin{align}
2 \E[  \| \nabla f(t)\| & \mathrm H_{\de A_u}^{\out} (t)  \mid f(t)=u] =  
\\ & = { \E\left[ 
\E\left[   \tr\left(\hess f\right)  - \hess f \left(v,v\right) \Big| f =u\right] 
\bigg| \frac{\nabla f }{\|\nabla f\|} = v  
\right]
}
\\
& = - u \ \E\left[  \tr( g^f)  -  g^f \left(v,v\right) \bigg| \frac{\nabla f }{\|\nabla f\|} = v  \right] \\
&= - u \ \E\left[    \tr( g^f)  -  g^f \left(\frac{\nabla f }{\|\nabla f\|},\frac{\nabla f }{\|\nabla f\|}\right)\right] \\
&= -u \left( \tr(g^f) - \E\left[\frac{\|\nabla f\|^2_f}{\|\nabla f\|^2}\right]\right),
\end{align}
where $\| \nabla f(t)\|_f$ denotes the norm in the metric $g^f$, see \cref{eq:ATmetric}. Since at any point $t\in M$ the Adler-Taylor metric $g^f_t$ of $f$ has strictly positive eigenvalues $a_1(t)$, $a_2(t)$, $a_3(t)$ with respect to $g_t$, we have that \begin{equation}
    \tr(g^f_t) - \E\left[\frac{\|\nabla f(t)\|^2_f}{\|\nabla f(t)\|^2}\right] \\ = a_1(t) + a_2(t) + a_3(t) - E_1(a_1(t),a_2(t),a_3(t)) .
\end{equation} 
Hence, \begin{equation}
I_1 = \frac{u e^{-u^2/2}}{\sqrt{8\pi^3}} \int_M \left( a_1+a_2+a_3 - E_1(a_1,a_2,a_3) \right) d\vol_g.
\end{equation}
As regards the second term, we may apply Fubini theorem and get \begin{equation}
    I_2 = \frac1{4\pi} \int_M \scal(t) \E 1_{A_u}(t) d\vol_g(t) = \frac1{4\pi} \left(1-\Phi(u)\right) \int_M\scal(t)d\vol_g(t).
\end{equation}
\end{proof}

\begin{proof}[Fourth part, $\E \mathcal L_2(A_u)$]
    Analogously to the previous proof, we apply 
    \cref{lem:kr} with $\alpha(t,X)=1$: \begin{equation}
        \E[\mathcal L_2^g (A_u)] = \frac{1}{2} \int_{SO(3)} \E[ \| \nabla f(t)\|  \mid f(t)=u] p_{f(t)}(u) d\vol_g (t). 
    \end{equation}
    Since $f$ is a Gaussian field with constant unit variance, it is uncorrelated, hence independent, of its derivatives. Then, we have \begin{equation}
        \E[ \| \nabla f(t)\|  \mid f(t)=u]=  \E[ \| \nabla f(t)\|].
    \end{equation}
    Since at any point $t\in M$ the Adler-Taylor metric $g^f_t$ of $f$ has strictly positive eigenvalues $a(t),b(t),c(t)$ with respect to $g_t$, we have that \begin{equation}
        \E[ \| \nabla f(t)\|] p_{f(t)}(t) = \frac{1}{\sqrt{2\pi}}e^{-u^2/2} E_2(a(t),b(t),c(t)).
    \end{equation}
    Summing up,
    \begin{equation} 
    \E[\mathcal L_2^g (A_u)] = \frac{1}{\sqrt{8\pi}}  e^{-u^2/2} \int_M E_2(a(t),b(t),c(t)) d\vol_g(t) .
    \end{equation}
    This concludes the proof.
\end{proof}

\subsection{Proof of \texorpdfstring{\cref{thm:main1}}{ }}
\begin{proof}[Proof of \texorpdfstring{\cref{thm:main1}}{ }]
Thanks to \cref{lem:metric}, the hypothesis of \cref{thm:main2} are satisfied with $M=SO(3)$ endowed with $g=g_{1,1}$ as in \cref{sec:geomSO3}, and $a_1(p)=a_2(p)=\xi^2$ and $a_3(p)=s^2$ at any $p\in SO(3)$. 
    Therefore, it is enough to compute the following quantities: $\vol(SO(3))$ and $\scal$, $E_1(\xi^2,\xi^2,s^2)$ and $E_2(\xi^2,\xi^2,s^2)$, and $\scal^f$; which are given in \cref{lem:geometrySO3}, \cref{thm:E1E2}, and \cref{lem:metric}. 
\end{proof}

\subsection{Proof of \texorpdfstring{\cref{lem:metric}}{}}
By \cref{eq:f} the Gaussian field $f$ is assumed to be almost surely a smooth function on $SO(3)$, so the first statement is immediate, see \cref{rem:onthereg}.
\begin{proof}[Proof of \texorpdfstring{\cref{eq:abc}}{ } and \texorpdfstring{\cref{eq:scalf}}{ }]
Provided that \cref{eq:metric} holds, the Gram matrix of $g^f$ at the point $(0,\frac\pi2,0)$ is $\Sigma_{(\xi,s)}$ as in \cref{eq:Sigmaxis}. Moreover, by \cref{lem:geometrySO3}, we have that $g=g_{1,1}$ and both metrics are left-invariant by \cref{lem:geometrySO3} and \cref{lem:covariance}. It follows that the eigenvalues of $g^f$ with respect to $g$ are constant. At the point $(0,\frac\pi2,0)$, the matrix of $g_{1,1}$ is $\Sigma_{(1,1)}=\mathbbm{1}_3$, hence the eigenvalues are the diagonal terms of the matrix $\Sigma_{(\xi,s)}$. This shows \cref{eq:abc}. For the proof of \cref{eq:scalf} we refer to \cref{lem:ricci}.
\end{proof}
We now pass to the proof of \cref{eq:metric} and, for the sake of readability, we denote by $\Sigma_{i,j}$ the $(i,j)$ component of $\Sigma_{(\xi,s)}$. 
We follow two different approaches. In the first part, we derive the coefficients by differentiating directly the field $f$, see \eqref{eq:f}, as a function of Euler angles coordinates. In the second part, we first write the covariance function of $f$ in terms of that of $X$, the complex field such that $f=\Re(f)$, and, in a second moment, we differentiate it.
    
\begin{proof}[Proof of \cref{eq:metric}: part 1, null terms and (3,3)]
        Recall \cref{eq:f}. In Euler angles coordinate, we can write $p=R(\f,\h,\psi)$ and express the components of the Gram matrix of $g^f$ as follows: \begin{equation}\Sigma_{i,j}:= \E\left[\de_i f\cdot\de_j f \right]=\E\left[ \sum_m \de_i T^l_{m,s}\cdot \sum_n\de_j T^l_{n,s} \right],
    \end{equation}    
where $\partial_i$, for $i=1,2,3$, denotes respectively $\frac{\partial}{\partial \f}$, $\frac{\partial}{\partial \h}$ and $\frac{\partial}{\partial \psi}$ and 
    \begin{align}
    T^l_{m,s}&:=\Re(\gamma^l_{m,s} D^l_{m,s})\\
& \ =\left\{\Re(\gamma^l_{m,s})\cos(m\f+s\psi)+\Im(\gamma^l_{m,s})\sin(m\f+s\psi)\right\}\cdot d^l_{m,s}(\h).
        \end{align}
The derivatives of the field are: \begin{align}
    \de_1 T_{m,s}^l  
    & = -m \left\{\Re(\gamma^l_{m,s})\sin(m\f+s\psi)-\Im(\gamma^l_{m,s})\cos(m\f+s\psi)\right\}\cdot d^l_{m,s}(\h) \\
    \de_2 T_{m,s}^l
    & = \left\{\Re(\gamma^l_{m,s})\cos(m\f+s\psi)+\Im(\gamma^l_{m,s})\sin(m\f+s\psi)\right\}\cdot (d^l_{m,s})'(\h) \\
    \de_3 T_{m,s}^l
    & =-s \left\{\Re(\gamma^l_{m,s})\sin(m\f+s\psi)-\Im(\gamma^l_{m,s})\cos(m\f+s\psi)\right\}\cdot d^l_{m,s}(\h).
    \end{align}
Since $\Re \gamma_{m,s}^l$, $\Im \gamma_{n,s}^l\sim \mathcal N\left( 0,\frac12\right)$ are independent for any choice of $m$ and $n$, the equalities follow: 
        \begin{align}
        & \E\left[  \de_3 T^l_{m,s}\cdot \de_3 T^l_{m,s} \right]  =  s^2  (d^l_{ms})^2; \\
        & \E\left[  \de_1 T^l_{m,s}\cdot \de_2 T^l_{m,s} \right] = \E\left[  \de_2 T^l_{m,s}\cdot \de_1 T^l_{m,s} \right] = 0 ;  \\
        & \E\left[  \de_2 T^l_{m,s}\cdot \de_3 T^l_{m,s} \right]  = \E\left[  \de_3 T^l_{m,s}\cdot \de_2 T^l_{m,s} \right] = 0.
        \end{align}
        Summing over $m=-l,\ldots,l$, we obtain the equalities
        \begin{equation}
            \Sigma_{12} = \Sigma_{21} = \Sigma_{23} = \Sigma_{32} = 0.
        \end{equation}
        To conclude, recalling the unitary property of D-Wigner matrices, that is $\sum_m  (d^l_{ms})^2=1$, see \cite[Eq. (3.29)]{libro}, we have $\Sigma_{33}=s^2$. 
        \end{proof}
        
Preliminary to the second part of the proof, we need to enumerate several properties of the covariance of the complex field $X$, such that $f=\Re X$. Let us consider two elements in $SO(3)$ parametrized by Euler angles as $p=R(\f_1,\h_1,\psi_1)$ and $q=R(\f_2,\h_2,\psi_2)$. By \cite[Proposition 40, Lemma 73]{GeoSpin2022}, we obtain that the covariance of the complex field can be expressed as follows: \begin{align}\label{eq:Gamma}
        \mathrm{Cov}(X(p),X(q))  
        &= k(\tilde \h) \ex^{is(\tilde \f + \tilde \psi)} \ex^{-is(\psi_1-\psi_2)} =:\Gamma(p^{-1}q) ,
        \end{align}
    where $p^{-1}$ denotes the inverse element of $p\in SO(3)$, and $\tilde \h$, $\tilde \f$ and $\tilde \psi$ are angles that depend on $\h_1$, $\h_2$, $\f_1-\f_2$, in such a way that $R(\tilde \f,\tilde \theta, \tilde \psi)=R_2(-\theta_1)R_3(\f_2-\f_1)R_2(\theta_2)$. 
\begin{remark}\label{rem:tilde}  Since $\tilde \h$ is a function of $\f_1-\f_2$, we obtain that $\frac{\de}{\de \f_2}\Gamma(p^{-1}q) = -\frac{\de}{\de \f_1}\Gamma(p^{-1}q)$. Therefore, we can write
    \begin{equation} 
            \frac{\de}{\de \f_1} \frac{\de}{\de \f_2} \Gamma(p^{-1}q)=-\frac{\de^2}{\de \f_1^2} \Gamma(p^{-1}q) , \qquad \frac{\de}{\de \psi_1} \frac{\de}{\de \f_2} \Gamma(p^{-1}q)=-\frac{\de}{\de \psi_1} \frac{\de}{\de \f_1} \Gamma(p^{-1}q).
    \end{equation}
From \cref{eq:Gamma}, for a fixed $\h$ we have the equalities \begin{align} 
    &\frac{\de^2}{\de \f_1^2} \Gamma(p^{-1}q) \bigg\rvert_{(\f_1,\h_1,\psi_1)=(\f_2,\h_2,\psi_2)} =  \frac{\de^2}{\de \f_1^2} \left[k(\tilde \h) \ex^{is(\tilde \f + \tilde \psi)}\big\rvert_{\h_1=\h_2=\h}\right]\bigg\rvert_{\f_1=\f_2}, \\
    &\frac{\de}{\de \psi_1} \frac{\de}{\de \f_1} \Gamma(p^{-1}q) \bigg\rvert_{(\f_1,\h_1,\psi_1)=(\f_2,\h_2,\psi_2)} = -is \cdot \frac{\de}{\de \f_1}\left[k(\tilde \h) \ex^{is(\tilde \f + \tilde \psi)}\big\rvert_{\h_1=\h_2=\h}\right]\bigg\rvert_{\f_1=\f_2} .
\end{align}
So, to compute the components $\Sigma_{i,j}$ for $(i,j)=(1,1),(1,3),(3,1)$, it is enough to find an expression for the first and second partial $\f_1$-derivatives of the function $k(\tilde \h) \ex^{is(\tilde \f + \tilde \psi)}\big\rvert_{\h_1=\h_2=\h}$ and their evaluation at $\f_1=\f_2$.
\end{remark}

\begin{remark}\label{rem:tilde3} 
    The circular covariance $k(\tilde\h)$, see \cref{eq:k}, is a function of $\cos\tilde\h$, since it measures the contribution to the covariance of the spherical distance from the North Pole. Hence, there exists a function $h$ such that $k(\tilde\h)=h(\cos\tilde\h)$. Moreover, 
    \begin{align} 
    \cos\tilde\h &=\sin\h_1 \sin\h_2\cos(\f_1-\f_2)+\cos\h_1\cos_2\\
    &=\sin^2\h\cos(\f_1-\f_2)+\cos^2\h
    \end{align}
    when $\h_1=\h_2=\h$. We compute their derivatives in special points.
\begin{lemma}[Derivatives of $k$]\label{lem:kappak}
    \begin{align}
        &\frac{\de^2}{\de \f_1^2} h(\cos\tilde\h)\big\rvert_{\f_1=\f_2,\h_1=\h_2=\h} = -h '(\cos\tilde\h) \sin^2\h \cos(\f_1-\f_2)\big\rvert_{\f_1=\f_2,\h_1=\h_2=\h} \\
        & =-h'(1) \sin^2\h   
        \\
        &\frac{\de}{\de \tilde\h} h(\cos\tilde\h)\big\rvert_{\tilde\h=0} = - h'(\cos\tilde\h) \sin \tilde\h \big\rvert_{\tilde\h=0} = 0
        \\
        &k''(0)=\frac{\de^2}{\de \tilde\h^2} h(\cos\tilde\h)\big\rvert_{\tilde\h=0} = - h'(\cos\tilde\h) \cos \tilde\h\big\rvert_{\tilde\h=0}=-h'(1)
    \end{align}
\end{lemma}
\begin{proof}
Immediate from the derivative of the composition.
\end{proof}
\end{remark}

\begin{remark}\label{rem:tilde4} 
    We can write \begin{equation}
    \ex^{is(\tilde \f + \tilde \psi)}= \left(\frac{\a^2}{\|\a\|^2}\right)^s 
    \end{equation} 
    where, as in \cite[Lemma 73]{GeoSpin2022}, we have 
    \begin{equation}
    \a=\cos\left(\frac{\tilde\h}{2}\right)\ex^{i(\tilde \f + \tilde \psi)}=\cos^2\left(\frac{\h}{2}\right) \ex^{-\frac{i}{2}(\f_1-\f_2)}+\sin^2\left(\frac{\h}{2}\right)\ex^{\frac{i}{2}(\f_1-\f_2)} 
    \end{equation}
\begin{lemma}[Derivatives of the phase of $\a$ and its $s$-th power]\label{lem:tilde}
    \begin{align}
         \frac{\de}{\de \f_1}& \left[\ex^{is(\tilde \f + \tilde \psi)}\big\rvert_{\h_1=\h_2=\h}\right]\bigg\rvert_{\f_1=\f_2} = -is\cos\h\\
        \frac{\de^2}{\de \f_1^2} &\left[\ex^{is(\tilde \f + \tilde \psi)}\big\rvert_{\h_1=\h_2=\h}\right]\bigg\rvert_{\f_1=\f_2}= -s^2\cos^2(\h)
    \end{align}
\end{lemma}
\begin{proof} 
We start by computing the derivative of $\ex^{i(\tilde \f + \tilde \psi)}= \left(\frac{\a^2}{\|\a\|^2}\right)$, keeping in mind that we will have to evaluate them at $\f_1=\f_2$: \begin{equation}
        \frac{\de}{\de \f_1} \left(\frac{\a^2}{\|\a\|^2}\right) = \frac{2\a\|\a\|^2\frac{\de \a}{\de \f_1} -\a^2\frac{\de \|\a\|^2}{\de \f_1}}{\|\a\|^4}=:G.
        \end{equation}
        To compute the derivatives of $\a$ and its square modulus we recall Remark~\ref{rem:tilde4}: \bega
        \frac{\de \a}{\de \f_1}\bigg\rvert_{\f_1=\f_2} &= \left[-\frac{i}{2} \cos^2\left(\frac{\h}{2}\right) \ex^{-\frac{i}{2}(\f_1-\f_2)}+\frac{i}{2}\sin^2\left(\frac{\h}{2}\right)\ex^{\frac{i}{2}(\f_1-\f_2)}\right]\bigg\rvert_{\f_1=\f_2}
        \\ &
        ={- \frac{i}{2}\cos\h};   \\
        \frac{\de^2 \a}{\de \f_1^2}\bigg\rvert_{\f_1=\f_2}  &= \left[-\frac{1}{4}\a \right]\bigg\rvert_{\f_1=\f_2} = -\frac{1}{4}; \\ 
        \frac{\de \|\a\|^2}{\de \f_1}\bigg\rvert_{\f_1=\f_2} &= \frac{\de }{\de \f_1} \left[\cos^2\left(\frac{\tilde\h}{2}\right)\right]\bigg\rvert_{\f_1=\f_2} 
        = \frac{1}{2}\frac{\de }{\de \f_1}\left[\cos\tilde\h\right]\bigg\rvert_{\f_1=\f_2} 
        \\ &
        = -\frac{1}{2}\left[\sin^2\h\sin(\f_1-\f_2)\right]\bigg\rvert_{\f_1=\f_2} = 0.
        \eega
        Hence, evaluating $G$ at $\f_1=\f_2$, recalling that $\a\rvert_{\f_1=\f_2}=1$, we may obtain: \begin{equation}
        G\rvert_{\f_1=\f_2} = -i\cos\h.
        \end{equation}
        On the other hand, to compute its first derivative, it is better to write it as follows \begin{equation}
        G= 2\frac{\a}{\|\a\|^2}\frac{\de \a}{\de \f_1} -\left(\frac{\a}{\|\a\|^2}\right)^2\frac{\de \|\a\|^2}{\de \f_1} =: G_1-G_2
        \end{equation}
        and compute derivatives of the two terms directly evaluating them. Indeed, \begin{align}
        \left[\frac{\de G_1}{\de \f_1}\right]\bigg\rvert_{\f_1=\f_2} &= \left[2 \frac{\de}{\de\f_1} \left(\frac{\a}{\|\a\|^2}\right) \cdot \frac{\de \a}{\de \f_1}+2  \frac{\a}{\|\a\|^2}\cdot \frac{\de^2 \a}{\de \f_1^2}\right]\bigg\rvert_{\f_1=\f_2}\\
        &=2(i\cos\h)^2+2\left(-\frac{1}{4}\right)\\
        &=-\frac{1}{2}(\cos^2\h+1)   \\
        \intertext{and}
        \left[\frac{\de G_2}{\de \f_1}\right]\bigg\rvert_{\f_1=\f_2} &= \left[2 \frac{\a}{\|\a\|^2} \frac{\de}{\de\f_1} \left(\frac{\a}{\|\a\|^2}\right) \frac{\de \a}{\de \f_1} + \left(\frac{\a}{\|\a\|^2}\right)^2\frac{\de^2 \|\a\|^2}{\de \f_1^2}\right]\bigg\rvert_{\f_1=\f_2}
        \\ &
        =-\frac{1}{2}\sin^2\h,
        \end{align} 
        where we have used that \begin{align}\left[ \frac{\de}{\de \f_1}\left(\frac{\a}{\|\a\|^2}\right)\right]\bigg\rvert_{\f_1=\f_2}& =\left[\frac{\frac{\de\a}{\de\f_1}\cdot \|\a\|^2-\a\cdot\frac{\de\|\a\|^2}{\de\f_1}}{\|\a\|^4}\right]\bigg\rvert_{\f_1=\f_2}=\left[\frac{\de\a}{\de\f_1}\right]\bigg\rvert_{\f_1=\f_2}
        \\ & 
        =-\frac{i}{2}\cos\h.
        \end{align}
        This implies that \begin{equation}\left[\frac{\de G}{\de \f_1}\right]\bigg\rvert_{\f_1=\f_2}=-\cos^2\h . \end{equation}
        Therefore, \begin{align}
        \frac{\de}{\de \f_1}\left[ \left(\frac{\a^2}{\|\a\|^2}\right)^s\right]\bigg\rvert_{\f_1=\f_2} &= \left[s \left(\frac{\a^2}{\|\a\|^2}\right)^{s-1} G\right]\bigg\rvert_{\f_1=\f_2} 
        =-is\cos\h,\\
        \intertext{and}
        \frac{\de^2}{\de \f_1^2} \left[\left(\frac{\a^2}{\|\a\|^2}\right)^s\right]\bigg\rvert_{\f_1=\f_2} &= \left[s(s-1)\left(\frac{\a^2}{\|\a\|^2}\right)^{s-1} G^2 + s \left(\frac{\a^2}{\|\a\|^2}\right)^{s-1} \frac{\de G}{\de \f_1}\right]\bigg\rvert_{\f_1=\f_2}\\
        &= -s(s-1)\cos^2\h -s\cos^2\h
        =-s^2\cos^2\h.
        \end{align}
        \end{proof}
            \end{remark}
We may conclude the proof of \cref{eq:metric}, and therefore of \cref{lem:metric}.
\begin{proof}[Proof of \cref{eq:metric}: part 2]
    Recall $p=R(\f_1,\h_1,\psi_1)$ and $q=R(\f_2,\h_2,\psi_2)$. Exchanging the differentiation with the integral, for $i,j=1,2,3$, we have 
    \begin{equation}
    \Sigma_{ij}(R(\f,\h,\psi)) = \de_i\de_{(j+3)} \E[f(R(\f_1,\h_1,\psi_1))f(R(\f_2,\h_2,\psi_2))]\big\rvert_{(\f_1,\h_1,\psi_1)=(\f_2,\h_2,\psi_2)},
        \end{equation} 
        where on the right hand side we are taking second derivatives of a function of $6$ variables, $(\f_1,\h_1,\psi_1,\f_2,\h_2,\psi_2)$.Writing $f=\frac{1}{2}(X+\overline X)$, we obtain 
        \begin{equation}
        \E[f(p)f(q)]=\frac{1}{2} \Re(\E[X(p)\overline{X(q)}])=\frac{1}{2}\Re(\Gamma(p^{-1}q)). 
        \end{equation}
        
        Recalling Remark~\ref{rem:tilde}, to compute $\Sigma_{11}$, $\Sigma_{13}$ and $\Sigma_{31}$, it is enough to compute: \begin{enumerate}
            \item $\frac{\de}{\de \f_1} \frac{\de}{\de \f_2} \Gamma(p^{-1}q) \mid_{(\f_1,\h_1,\psi_1)=(\f_2,\h_2,\psi_2)}$;
            \item $\frac{\de}{\de \psi_1} \frac{\de}{\de \f_2} \Gamma(p^{-1}q) \mid_{(\f_1,\h_1,\psi_1)=(\f_2,\h_2,\psi_2)}$.
        \end{enumerate}
        Then, we can write \begin{align} 
         (1) &=-\frac{\de^2}{\de \f_1^2} \left[h(\cos\tilde \h) \ex^{is(\tilde \f + \tilde \psi)}\big\rvert_{\h_1=\h_2=\h}\right]\bigg\rvert_{\f_1=\f_2}\\
         &=-\frac{\de^2}{\de \f_1^2} \left[h(\cos\tilde \h)\big\rvert_{\h_1=\h_2=\h}\right]\bigg\rvert_{\f_1=\f_2}
         -h(1) \frac{\de^2}{\de \f_1^2} \left[\ex^{is(\tilde \f + \tilde \psi)}\big\rvert_{\h_1=\h_2=\h}\right]\bigg\rvert_{\f_1=\f_2}
         \\
         & \quad-2\frac{\de}{\de \f_1} \left[h(\cos\tilde \h)\big\rvert_{\h_1=\h_2=\h}\right]\bigg\rvert_{\f_1=\f_2} \times \frac{\de}{\de \f_1} \left[\ex^{is(\tilde \f + \tilde \psi)}\big\rvert_{\h_1=\h_2=\h}\right]\bigg\rvert_{\f_1=\f_2} 
        \\
         &= \sin^2(\h) h'(1) +h(1) s^2 \cos^2(\h)
        \end{align}
        where the last inequality follows from Remark~\ref{rem:tilde4} and Lemma~\ref{lem:tilde}. Analogously, \begin{align}
        (2) &= is \cdot \frac{\de}{\de \f_1}\left[h(\cos\tilde \h) \ex^{is(\tilde \f + \tilde \psi)}\big\rvert_{\h_1=\h_2=\h}\right]\bigg\rvert_{\f_1=\f_2} \\
        &= is h(1) \frac{\de}{\de \f_1}\left[\ex^{is(\tilde \f + \tilde \psi)}\big\rvert_{\h_1=\h_2=\h}\right]\bigg\rvert_{\f_1=\f_2}\\
        &=s^2 h(1) \cos\h.
        \end{align}
        We conclude by noting that $h(1)=k(0)$ and $h'(1)=-k''(0)$. To compute the $(2,2)$ term, the last one, we notice that \begin{align} 
                \frac{\de}{\de \h_1} \frac{\de}{\de \h_2} \Gamma(p^{-1}q) \bigg\rvert_{(\f_1,\h_1,\psi_1)=(\f_2,\h_2,\psi_2)} &= \frac{\de}{\de \h_1} \frac{\de}{\de \h_2} \left[k(\tilde \h) \ex^{is(\tilde \f + \tilde \psi)}\big\rvert_{\f_1=\f_2}\right]\bigg\rvert_{\h_1=\h_2} \\
                &=\frac{\de}{\de \h_1} \frac{\de}{\de \h_2} \left[k(\h_1-\h_2)\right]\bigg\rvert_{\h_1=\h_2} \\
                &= -k''(0).
            \end{align} 
        \end{proof}



\appendix

\renewcommand{\thesection}{\Alph{section}} 
\makeatletter
\renewcommand\@seccntformat[1]{\appendixname\ \csname the#1\endcsname.\hspace{0.5em}}
\makeatother

\section{Curvatures of the Adler-Taylor metric}\label{appA}

\begin{remark}
Recall the matrix $\Sigma_{\xi,s}(\h)$ in \cref{eq:intro_norm}.  
In the following we will need its inverse matrix, which can be easily computed as \begin{align}\label{equ:SigmaF-1}
        \left(\Sigma_{\xi,s}(\theta)\right)^{-1} = \frac{1}{\det \left(\Sigma_{\xi,s}(\theta)\right)} \adj\left(\Sigma_{\xi,s}(\theta)\right)^T= \begin{pmatrix}
            \frac{1}{\xi^2\sin^2(\h)} & 0 & -\frac{\cos(\h)}{\xi^2\sin^2(\h)} \\
            0 & \frac{1}{\xi^2} & 0 \\
            -\frac{\cos(\h)}{\xi^2\sin^2(\h)} & 0 & \frac{1}{s^2}+\frac{1}{\xi^2 \tan^2(\h)}
            \end{pmatrix},
\end{align}
where $\adj(M)$ denotes the cofactor matrix of $M$.
\end{remark}
\subsection{Christoffel symbols}
\begin{lemma}
    The Christoffel symbols of the metric $g^f$ in Euler angles coordinates are:
    \begin{align} \label{equ:Gamma1}
        {\Gamma}_{ij}^1 &= \begin{pmatrix}
             0 & \frac{\cos\h}{\sin\h}\left(1-\frac{s^2}{2\xi^2}\right) & 0 \\
            (*)  & 0 &  -\frac{1}{\sin \h} \frac{s^2}{2\xi^2} \\
            0 & (*) & 0
        \end{pmatrix}  =  \begin{pmatrix}
             0 & \cos\h\left(\frac{1}{\sin\h} + \Gamma_{23}^1\right) & 0 \\
            (*)  & 0 &  -\frac{1}{\sin \h} \frac{s^2}{2\xi^2} \\
            0 & (*) & 0
        \end{pmatrix};\\
        \label{equ:Gamma2}
        {\Gamma}_{ij}^2 &= \begin{pmatrix}
            -\frac{ \sin(2\h)}{2} \left( 1 - \frac{s^2}{\xi^2}\right) & 0 & \frac{s^2 \sin\h}{{2}\xi^2} \\
            0 & 0 & 0 \\
            (*) & 0 & 0
        \end{pmatrix};\\
        \label{equ:Gamma3}
        {\Gamma}_{ij}^3 &= \begin{pmatrix}
            0 & -\left(1-\frac{s^2}{2\xi^2}\right)\frac{\cos^2\h}{\sin\h} - \frac{\sin\h}{2} & 0 \\
            (*) & 0 & \frac{\cos\h}{\sin\h}\frac{s^2}{2\xi^2} \\
            0 & (*)  & 0
        \end{pmatrix} \\
        & = \begin{pmatrix}
            0 & -\cos\h \cdot\Gamma_{12}^1 - \frac{\sin\h}{2} & 0 \\
            (*) & 0 & -\cos\h \cdot \Gamma_{23}^1 \\
            0 & (*)  & 0
        \end{pmatrix} = -\cos\h \cdot \Gamma^1_{ij} -\frac{\sin\h}{2} \delta_{12,21};
    \end{align}
    where the $(*)$ are such that each matrix is symmetric and $\delta_{12,21}$ is a $3\times 3$ matrix having non-null entries only in $12$ and $21$, equal to $1$.
\end{lemma}

\begin{proof} Let us denote the matrix \eqref{eq:metric} as $g_{ij}$ and its inverse \eqref{equ:SigmaF-1} accordingly as $g^{ij}:=\left(\Sigma_{ij}\right)^{-1}$. The general formula for the Christoffel symbols in Euler angles coordinates, see \cref{def:eulerangles} and \cite{leeriemann}, is  \begin{equation}\label{equ:genChrSym}
    \Gamma_{ij}^l = \frac{1}{2} g^{kl} \left(  \frac{\de g_{jk}}{\de x^i}+\frac{\de g_{ki}}{\de x^j}-\frac{\de g_{ij}}{\de x^k}\right),
\end{equation} 
where \begin{equation}
    \frac{\de}{\de x^1} = \frac{\de}{\de \f}, \quad \frac{\de}{\de x^2} = \frac{\de}{\de \h}, \quad \frac{\de}{\de x^3} = \frac{\de}{\de \psi}.
\end{equation}
We will exploit several times that 
\begin{equation}\label{equ:propG} \begin{cases}
    \frac{\de g_{ij}}{\de x^1}\equiv\frac{\de g_{ij}}{\de x^3}\equiv0 & \forall i,j \in \{1,2,3\}, \\
    g_{ij}\equiv g^{ij} \equiv 0 & \forall |i-j|=1.
\end{cases}\end{equation} 
Moreover, we have that \begin{align}\label{equ:derivG}
\frac{\de g_{11}}{\de x^2} &= \left(\xi^2-s^2\right)\sin(2\h),\\
\frac{\de g_{13}}{\de x^2} &= -s^2\sin\h,\\
\frac{\de g_{22}}{\de x^2} &= \frac{\de g_{33}}{\de x^2} = 0.
\end{align} 
Then, \begin{equation}
    \mathbf{\Gamma_{ij}^1} = \frac{1}{2\xi^2\sin (2\h)} \left\{\left(\frac{\de g_{j1}}{\de x^i}+\frac{\de g_{1i}}{\de x^j}\right)-\cos\h \left(\frac{\de g_{j3}}{\de x^i}+\frac{\de g_{3i}}{\de x^j}\right)\right\}.
\end{equation}
Therefore, \begin{align}
 {\Gamma_{11}^1} & = \Gamma_{22}^1 = \Gamma_{33}^1 = \Gamma_{13}^1 = \Gamma_{31}^1 = 0, \\
 \Gamma_{12}^1 & = \Gamma_{21}^1 = \frac{1}{2\xi^2\sin^2 (\h)} \left\{ \frac{\de g_{11}}{\de x^2} - \cos\h \frac{\de g_{13}}{\de x^2}\right\} = \frac{\cos\h}{\sin\h}\left(1-\frac{s^2}{2\xi^2}\right), \\
 \Gamma_{23}^1 & = \Gamma_{32}^1 = \frac{1}{2\xi^2\sin^2 (\h)} \left\{ \frac{\de g_{31}}{\de x^2} - \cos\h \frac{\de g_{33}}{\de x^2}\right\} = -\frac{1}{\sin \h} \frac{s^2}{2\xi^2}.
\end{align}

Analogously, we have: \begin{align}
    \mathbf{\Gamma_{ij}^3} & = -\frac{\cos\h}{2\xi^2\sin^2(\h)} \left(\frac{\de g_{j1}}{\de x^i}+\frac{\de g_{1i}}{\de x^j}\right) + \frac{1}{2}\left(\frac{1}{s^2}+\frac{1}{\xi^2\tan^2\h}\right) \left( \frac{\de g_{j3}}{\de x^i}+\frac{\de g_{3i}}{\de x^j}\right);
    \\
    {\Gamma_{11}^3} & = \Gamma_{22}^3 = \Gamma_{33}^3 = \Gamma_{13}^3 = \Gamma_{31}^3 = 0, 
    \\
        {\Gamma_{12}^3} & = \Gamma_{21}^3 = -\frac{\cos\h}{2\xi^2\sin^2(\h)} \frac{\de g_{11}}{\de x^2} + \frac{1}{2}\left(\frac{1}{s^2}+\frac{1}{\xi^2\tan^2\h}\right) \frac{\de g_{31}}{\de x^2}  ,
        \\
        & \hspace{5cm} = -\left(1-\frac{s^2}{2\xi^2}\right)\frac{\cos^2\h}{\sin\h} - \frac{\sin\h}{2};
        \\
        \Gamma_{23}^3 & = \Gamma_{32}^3 = -\frac{\cos\h}{2\xi^2\sin^2(\h)} \frac{\de g_{31}}{\de x^2} = \frac{\cos\h}{\sin\h} \frac{s^2}{2\xi^2},
        \\
    \mathbf{\Gamma_{ij}^2} & = \frac{1}{2\xi^2}\left( \frac{\de g_{j2}}{\de x^i}+\frac{\de g_{2i}}{\de x^j}-\frac{\de g_{ij}}{\de x^2}\right);
    \\
        {\Gamma_{12}^2} & = \Gamma_{21}^2 = \Gamma_{23}^2 = \Gamma_{32}^2 = 0, \\
        {\Gamma_{11}^2} & = -\frac{1}{2\xi^2} \frac{\de g_{11}}{\de x^2} = -\frac{ \sin(2\h)}{2} \left( 1 - \frac{s^2}{\xi^2}\right),\\
        {\Gamma_{22}^2} & = \frac{1}{2\xi^2} \frac{\de g_{22}}{\de x^2} = 0, \\
        {\Gamma_{33}^2} & =  -\frac{1}{2\xi^2} \frac{\de g_{33}}{\de x^2} = 0, \\
        \Gamma_{13}^2 & = \Gamma_{31}^2 = -\frac{1}{2\xi^2} \frac{\de g_{13}}{\de x^2} = \frac{s^2 \sin\h}{2\xi^2}.
    \end{align}
\end{proof}

\subsection{Riemann tensor}

\begin{lemma}\label{lem:riemann04} Let us represent the coordinates of the Riemann tensor of type $(0,4)$ in a (symmetric) matrix: \begin{equation}
    R^f_{ijkl}=\begin{pmatrix}
        -\sin^2\h \left(\xi^2-\frac{3s^2}{4}\right) -\cos^2\h \frac{s^4}{4\xi^2} & 0 & \cos\h \frac{s^4}{4\xi^2} \\
        0 & -\sin^2\h\frac{s^4}{4\xi^2} & 0 \\
        \cos\h \frac{s^4}{4\xi^2} & 0 & -\frac{s^4}{4\xi^2}
    \end{pmatrix},
\end{equation}
where we adopt the lexicografic order over the pairs $(i,j),(k,l)\in \{(1,2),(1,3),(2,3)\}$.
\end{lemma}

\begin{lemma}\label{lem:riemann13} The Riemann tensor of type $(1,3)$ has the following coordinates.
\begin{align}
    {{R_{131}^1}} &= {-{R_{311}^1}} = {{R_{231}^2}} = {-{R_{321}^2}} =
    -{{R_{123}^2}} = {{R_{213}^2}} =-{ {R_{133}^3}} = { {R_{313}^3}} = { \cos \h  \left(\frac{s^2}{2\xi^2}\right)^2} \\
    {  {R_{133}^1}} &= {-  {R_{313}^1}} = { {R_{233}^2}} = {- {R_{323}^2}} = \left(\frac{s^2}{2\xi^2}\right)^2 \\
    { {R_{232}^3}} &= {- {R_{322}^3}} = -\frac{s^2}{4\xi^2} \\
    {  {R_{122}^1}} &= {-  {R_{212}^1}} =  \left(1-\frac{s^2}{2\xi^2}\right) \frac{1}{\sin\h} - \left(1-\frac{s^2}{4\xi^2}\right) \\
    { {R_{121}^2}} &= {- {R_{211}^2}} = -\sin^2\h \cdot \left(1-\frac{3s^2}{4\xi^2}\right) -\cos^2\h \cdot \left(\frac{s^2}{2\xi^2}\right)^2 \\
    { {R_{131}^3}} &= {- {R_{311}^3}} = -\sin^2\h\frac{s^2}{4\xi^2} -\cos^2\h \cdot \left(\frac{s^2}{2\xi^2}\right)^2.
\end{align}
\end{lemma}

\begin{proof}[Proof]
    Let us recall the general formula: \begin{equation}
    R_{ijk}^m = \frac{\de \Gamma^m_{jk}}{\de x^i} - \frac{\de \Gamma^m_{ik}}{\de x^j} + \Gamma^m_{ih}\Gamma^h_{jk} - \Gamma^m_{jh}\Gamma^h_{ik}.
\end{equation}
We highlight some useful relations and derivatives: \begin{align}
    \Gamma_{12}^1 &= \frac{\cos\h}{\sin\h}+\cos\h \cdot \Gamma_{23}^1, \qquad  &\Gamma_{11}^2 &= -\sin\h\cos\h +2\cos\h\cdot \Gamma_{13}^2, \\
    \Gamma_{23}^3 &= -\cos\h \cdot\Gamma_{23}^1, \qquad &\Gamma_{12}^3 &= -\cos\h \cdot \Gamma_{12}^1 -\frac{\sin\h}{2},\\
    \frac{\de \Gamma^1_{12}}{\de x^2} &= -\left(1-\frac{s^2}{2\xi^2}\right) \frac{1}{\sin^2\h},  & \frac{\de \Gamma^1_{23}}{\de x^2} &= \frac{s^2}{2\xi^2}\frac{\cos\h}{\sin^2\h}, \\
    \frac{\de \Gamma^1_{11}}{\de x^2} &= (1-2\cos^2\h) \left(1-\frac{s^2}{\xi^2}\right), &
    \frac{\de \Gamma^2_{13}}{\de x^2} &= \frac{s^2\cos\h}{2\xi^2}.
\end{align}

Since
\begin{equation}
    \mathbf{R_{ijk}^1} = \frac{\de \Gamma^1_{jk}}{\de x^i} - \frac{\de \Gamma^1_{ik}}{\de x^j} + \Gamma^1_{ih}\Gamma^h_{jk} - \Gamma^1_{jh}\Gamma^h_{ik},
\end{equation}
we have \begin{align}
    R_{131}^1 &= \Gamma_{12}^1\Gamma_{13}^2 - \Gamma_{23}^1\Gamma_{11}^2 = \underbrace{\frac{\cos\h}{\sin\h} \Gamma_{13}^2  + \cos\h\sin\h \cdot \Gamma_{23}^1}_{=0} -\cos\h\cdot \Gamma_{13}^2\Gamma_{23}^1 \\
    &= \left(\frac{s^2}{2\xi^2}\right)^2 \cos\h, \\
    R_{133}^1 &= -\Gamma_{23}^1\Gamma_{13}^2 =\left(\frac{s^2}{2\xi^2}\right)^2 ,\\
    R_{122}^1 &= -\frac{\de \Gamma_{12}^1}{\de x^2} -\frac{\sin\h}{2\cos\h} \Gamma_{12}^1 -\frac{1}{2}= \left(1-\frac{s^2}{2\xi^2}\right)\frac{1}{\sin^2\h} - \left(1-\frac{s^2}{4\xi^2}\right),\\
    R_{322}^1 &=  - \frac{\de \Gamma^1_{32}}{\de x^2} - \Gamma^1_{21}\Gamma^1_{32} - \Gamma^1_{23}\Gamma^3_{32} = - \frac{\de \Gamma^1_{32}}{\de x^2} - \frac{\cos\h}{\sin\h} \Gamma_{23}^1 = 0.
    \end{align}
    The missing ones follow immediately by the properties of $\Gamma$ and its  symmetries.
Analogously, we have:
\begin{align}
    \mathbf{R_{ijk}^2} & = \frac{\de \Gamma^2_{jk}}{\de x^i} - \frac{\de \Gamma^2_{ik}}{\de x^j} + \Gamma^2_{ih}\Gamma^h_{jk} - \Gamma^2_{jh}\Gamma^h_{ik}, 
    \\
    R_{132}^2 &= -\sin\h\cos\h\cdot\Gamma_{23}^1 -\frac{\cos\h}{\sin\h}\Gamma_{13}^2  = 0, \\
    R_{231}^2 &= \frac{\de \Gamma_{13}^2}{\de x^2} -\Gamma_{13}^2 \Gamma_{12}^1 = \left(\frac{s^2}{2\xi^2}\right)^2 \cos\h ,\\
    R_{233}^2 &= -\Gamma_{13}^2\Gamma_{23}^1 = \left(\frac{s^2}{2\xi^2}\right)^2 ,\\
    R_{121}^2 &= -\frac{\de \Gamma_{11}^2}{\de x^2} -\cos\h \cdot \Gamma_{13}^2\Gamma_{12}^1-\frac{\sin\h}{2} \Gamma_{13}^2 +\Gamma_{11}^2\Gamma_{12}^1 ,\\
    R_{123}^2 &= -\frac{\de \Gamma^2_{13}}{\de x^2} -\sin\h\cos\h \Gamma_{23}^1 +\cos\h \Gamma_{13}^2\Gamma_{23}^1 , 
    \\
    \mathbf{R_{ijk}^3} & = \frac{\de \Gamma^3_{jk}}{\de x^i} - \frac{\de \Gamma^3_{ik}}{\de x^j} + \Gamma^3_{ih}\Gamma^h_{jk} - \Gamma^3_{jh}\Gamma^h_{ik}, 
    \\
    {R_{232}^3} &= \cos\h\left(- \frac{\de \Gamma^1_{32}}{\de x^2} - \frac{\cos\h}{\sin\h} \Gamma_{23}^1\right) +\Gamma_{23}^1\frac{\sin\h}{2} = \cos\h R_{322}^1  +\Gamma_{23}^1\frac{\sin\h}{2} = \Gamma_{23}^1\frac{\sin\h}{2},\\
    R_{131}^3 &= -\cos\h \cdot R_{131}^1-\frac{\sin\h}{2}\cdot \Gamma_{13}^2 ,\\
    R_{133}^3 &= -\Gamma_{23}^3 \Gamma_{13}^2.
\end{align}
We omit the term ${R_{122}^3}$ because it's not needed to compute the coordinates of the tensor $R_{ijkl}$.
\end{proof}

\begin{proof}[Proof of Lemma~\ref{lem:riemann04}]
    Let us recall the general formula \begin{equation}
        R_{ijkl} = R_{ijk}^m g_{lm},
    \end{equation}
    where $g_lm= \Sigma_{lm}(\h)$, recall \eqref{eq:metric}. Hence, exploiting symmetries and the previous calculations,  we obtain \begin{align}
        R_{1212} &= \xi^2 R_{121}^2 ,\\
        R_{1313} &= s^2\cos\h R_{131}^1 + s^2 R_{131}^3 ,\\
        R_{2323} &= R_{232}^3 g_{33}, \\
        R_{1213} &= R_{121}^1 g_{31} + R_{121}^3 g_{33} = 0,\\
        R_{1223} &= - R_{1232} = -\xi^2 R_{123}^2, \\
        R_{1323} &= -R_{1332} = -\xi^2 R_{133}^2 = 0.
    \end{align}

\end{proof}

\subsection{Scalar curvature and sectional curvatures}
\begin{remark}
Note that for hypersurfaces of $\R^n$, the scalar curvature $\scal$, see \cref{eq:scal}, is equal to twice the Gaussian curvature. In particular, for a $n$-sphere of radius $r$, we have $\scal\equiv n(n-1)r^{-2}$.
\end{remark}
\begin{lemma}\label{lem:ricci}
    The scalar curvature of the metric $g^f$ is constant and equals \begin{equation}
\scal^f = \frac{2}{\xi^2} - \frac{s^2}{2\xi^4}.
\end{equation}
\end{lemma}

\begin{proof}
  Recall the metric \eqref{eq:metric} and its inverse \eqref{equ:SigmaF-1}. Since $R_{1223}=-\cos\h R_{2323}$, we have \begin{align}
   &\scal = R_{ijkl}g^{il}g^{kj} \\
   &= 2g^{22} \left(-R_{1212}g^{11}-R_{2323}g^{33}+2R_{1223}g^{13}\right) +2 R_{1313} \left((g^{13})^2-g^{11}g^{33}\right) \\
    &= 2g^{22} \left(-R_{1212}g^{11}-R_{2323}(g^{33}+2\cos\h g^{13})\right) +2 R_{1313} \left((g^{13})^2-g^{11}g^{33}\right) \\
    &= 2g^{22} \left(-R_{1212}g^{11}-R_{2323}\left( \frac{1}{s^2} - \frac{\cos^2\h}{\xi^2\sin^2 \h}\right)\right) -2 R_{1313} \frac{1}{s^2\xi^2\sin^2\h } \\
    &=\frac{2}{\xi^2} \left(\frac{1}{\xi^2\sin^2\h} \left(\sin^2\h\left(\xi^2-\frac{3s^2}{4}\right) + \cos^2\h \frac{s^4}{4\xi^2}\right) + \frac{s^4}{4\xi^2}\left( \frac{1}{s^2} - \frac{\cos^2\h}{\xi^2\sin^2 \h}\right)  \right) \\
    & \hspace{6cm} +2 \frac{s^4\sin^2\h}{4\xi^2} \frac{1}{s^2\xi^2\sin^2\h}\\
    &=\frac{2}{\xi^2} \left(\left(1-\frac{3s^2}{4\xi^2}\right) + { \frac{\cos^2\h}{\sin^2\h} \frac{s^4}{4\xi^4}}+ \frac{s^2}{4\xi^2}{ -\frac{\cos^2\h}{\sin^2\h} \frac{s^4}{4\xi^4}} \right) +\frac{s^2}{2\xi^4}\\
    &= \frac{2}{\xi^2} \left(1-\frac{s^2}{2\xi^2}\right)+\frac{s^2}{2\xi^4}.
\end{align}
Recalling that $\tr(R) = -\frac{1}{2}\scal$, we conclude.
\end{proof}
\begin{lemma}
Let us denote by $\sigma(ij)$ the coordinate plane generated by the $i$th and $j$th coordinate direction. The sectional curvatures of the coordinate planes are \begin{align}
        \mathrm{Sec}(\sigma(12)) &= \frac{\sin^2\h \left(\xi^2-\frac{3s^2}{4}\right) +\cos^2\h \frac{s^4}{4\xi^2}}{\sin^2\h \xi^4+ \cos^2\h s^2\xi^2}, \\
        \mathrm{Sec}(\sigma(13)) &= \frac{s^2}{4\xi^4}, \\
        \mathrm{Sec}(\sigma(23)) &= \frac{s^2}{4\xi^4}.
    \end{align}
\end{lemma}

\subsection{The Lipschitz-Killing curvatures of \texorpdfstring{$SO(3)$}{SO(3)} in the Adler-Taylor metric}

\begin{proposition}\label{pr:lkSO3} Let us denote by $\LK_j^f(SO(3))$ the $j^{th}$ Lipschitz-Killing curvature of $SO(3)$, see \cref{subsec:setting}, computed with respect to the Adler-Taylor metric of a spin random field $f$, see \cref{eq:ATmetric}, dependent on two parameters $\xi^2>0$ and $s\in \mathbb Z$ defined in \cref{thm:main1}. The following equalities are satisfied: \begin{align}
    \mathcal L_0^f(SO(3)) &= \chi(SO(3)) = 0, \\
    \mathcal L_1^f(SO(3)) &= 4|s|\pi \left(1-\frac{s^2}{4\xi^2} \right),\\
    \mathcal L_2^f(SO(3)) &= 0, \\
    \mathcal L_3^f(SO(3)) &= \vol^f(SO(3)) = 8\pi^2 \xi^2 |s| .
\end{align}
\end{proposition}
\begin{proof} By \cite[Eq. (7.6.2)]{AdlerTaylor} the curvatures with even indices are zero, and the third one is  \begin{equation}
    \mathcal L_3^f(SO(3)) = \vol^f(SO(3)) = \sqrt{\det \Sigma_{(\xi,s)}} \vol (SO(3)) = 8 \pi^2 \xi^2 |s|,
\end{equation}
see \cref{eq:Sigmaxis} and \cref{lem:geometrySO3}. Applying Lemma~\ref{lem:ricci}, we get \begin{align}
        \mathcal L_1^f(SO(3)) 
        &= -\frac{1}{2\pi} \int_{SO(3)} \Tr^{TSO(3)}(R^f) \ dV^f =   -\frac{1}{2\pi} \left(\frac{s^2}{4\xi^4} - \frac{1}{\xi^2}\right) \vol^f(SO(3)) \\
        & = -4 |s|\pi \left(\frac{s^2}{4\xi^2} - 1\right).
    \end{align}
\end{proof}

\section{Proof of \texorpdfstring{\cref{thm:E1E2}}{}}\label{appB}
\begin{proof}[Proof of \texorpdfstring{\cref{thm:E1E2}}{}]
Recall \cref{def:E1E2}. For any value of $\xi>0$ and $s\neq 0$, we have \begin{equation}
    E_1\left(\xi^2,\xi^2,s^2\right) = \xi^2 E_1\left(1,1,\frac{s^2}{\xi^2}\right), \quad E_2\left(\xi^2,\xi^2,s^2\right) = \xi E_2\left(1,1,\frac{s^2}{\xi^2}\right).
\end{equation} 
We define $a = \frac{\xi^2}{s^2}$. Hence, we can equivalently show that, for any $a>0$, we have: \begin{align}
    & E_1\left( 1,1,\frac1a \right) = \begin{cases}
         1+\frac1a \left( 1 -  \frac1{2\sqrt{1-\frac1a}} \left( \log a + 2 \log \left( 1 + \sqrt{1-\frac1a}\right) \right) \right)   &\text{ if } a>1, \\
         1 + \frac1a \left( 1 -  \frac{\arctan \sqrt{\frac1a -1 }}{\sqrt{\frac1a - 1}}   \right) &\text{ if } a<1;
    \end{cases}  \\
    & E_2\left( 1,1,\frac1a \right) = \begin{cases}
        \sqrt{\frac2{\pi}} \left( \sqrt{\frac1{a}} +\frac{\arcsin{\sqrt{1-\frac1a}}}{\sqrt{1-\frac1a}} \right) &\text{ if } a>1, \\
        \sqrt{\frac2{\pi}} \left( \sqrt{\frac1{a}} +\frac{\arcsinh {\sqrt{\frac1a-1}}}{\sqrt{\frac1a-1}} \right) &\text{ if } a<1.
    \end{cases}
\end{align}
Observe that we can substitute $\gamma$ with $\frac{\gamma}{|\gamma|}$, which is uniformly distributed on the sphere $S^2$, and that $\frac{\gamma}{|\gamma|}$ is independent of $|\gamma|$. Then, taking polar coordinates we have \begin{align}
    E_2\left( 1,1,\frac1a\right) 
    &= \E[|\gamma|] \frac{1}{4\pi}\int_0^{2\pi} d\f \int_0^\pi d\h \sqrt{\sin^2\theta+\frac{1}{a}\cos^2\theta} \sin\theta  \\
    &= \E[|\gamma|]  \int_0^{\pi/2} \sqrt{1-\cos^2\h\left(1-\frac1a\right)} \sin\h d\h.
    \intertext{If $a>1$, substituting $x=\sqrt{1-\frac1a}\cos\h$ and since $\int_0^{\beta} \sqrt{1-x^2} dx = \frac12 \left( \sqrt{1-\beta^2}\beta + \arcsin\beta \right)$ when $0\le \beta\le 1$, we get}
     &= \frac{\E\left[|\gamma|\right]}2 \left( \frac1{\sqrt{a}} +\frac{\arcsin{\sqrt{1-\frac1a}}}{\sqrt{1-\frac1a}} \right).
     \intertext{On the converse, when $a<1$, substituting $x=\sqrt{\frac1a-1}\cos\h$ and since $\int_0^{\beta} \sqrt{1+x^2} dx = \frac12 \left( \sqrt{1+\beta^2}\beta + \arcsinh\beta \right)$ when $0\le \beta\le 1$, we have}
    & = \E[|\gamma|] \left( \frac12 + \frac{\arcsinh \sqrt{\frac1a-1}}{2\sqrt{\frac1a-1}} \right).
\end{align}
Note that $\E |\gamma| = \E \chi_3  = 2\sqrt{\frac2\pi}$, where $\chi_3$ denotes a $\chi$ random variable with $3$ degree of freedom. 
Regarding $E_1$, we can start the computation as we have done for $E_2$ and get 
\begin{align}
    E_1\left(1,1,\frac1a\right)
&=\frac{1}{4\pi}\int_0^{2\pi}\int_0^\pi \frac{\sin^2\theta+\frac{1}{a^2}\cos^2\theta}{\sin^2\theta+\frac{1}{a}\cos^2\theta} \sin\theta d\theta d\varphi
\\
&= 1 + \left( \frac1{a^2}-\frac1a \right) \int_0^{\pi/2}  \frac{\cos^2\h}{1 -\left( 1 - \frac1a \right) \cos^2\h} \sin\h d\h.
\intertext{Then, we substitute $x=\cos \h$, and add and subtract $1$, to obtain} 
&= 1 + \frac1a - \frac1a \int_0^1 \frac1{1 -\left( 1 - \frac1a \right) x^2} dx.
\intertext{Now, recall that $\int_0^1 \frac1{1+\beta x^2} dx = \frac{\arctan \sqrt \beta}{\sqrt \beta}$ for any $\beta>0$, and $\int_0^1 \frac1{1-\beta x^2} dx = \frac1{2\sqrt \beta}\log \frac{1+\sqrt\beta}{1-\sqrt\beta}$ for any $0< \beta<1 $. Then, when $a>1$ we apply the previous to $\beta=1-\frac1a$ to have }
& = 1 + \frac1a - \frac1a \frac1{2\sqrt{1-\frac1a}} \log \frac{1+\sqrt{1-\frac1a}}{1-\sqrt{1-\frac1a}}.
\intertext{On the converse, when $a<1$, then $\beta = \frac1a - 1$ and }
& = 1 + \frac1a - \frac1a \frac{\arctan \sqrt{\frac1a -1 }}{\sqrt{\frac1a - 1}}.
\end{align}
Rearranging the previous expressions, we conclude.
\end{proof}

\section{Lipschitz-Killing curvatures}\label{appC}

\begin{proof}[Proof of \cref{eq:LK3dim}]
    We recall and explain term-by-term \cite[Definition 10.7.2]{AdlerTaylor}. Let us consider a Borel subset $B$ of a manifold $(A^3,g)$, possibly with boundary. Its $i$th Lipschitz-Killing curvature is \begin{multline}
        \mathcal L_i(A, B) = \sum_{j=i}^3 (2\pi)^{-\frac{j-i}2} \sum_{m=0}^{\left\lfloor \frac{j-i}2 \right\rfloor} \frac{(-1)^m C(3-j,j-i-2m)}{m!(j-i-2m)!} 
        \\ \times 
        \int_{\partial_j A\cap B} \int_{\mathbb S(T_t\partial_j A^\perp)} \Tr^{T_t\partial_j A^\perp}  \left(  R^m  S^{j-i-2m}_{\nu}\right) 
        \alpha(\nu) \mathcal H_{3-j-1}(d \nu) \mathcal H_j(dt);
    \end{multline}
    where: \begin{itemize}
        \item $\partial_j A$ denotes the $j$-dimensional boundary of $A$, which is a disjoint union of a finite number of $j$th dimensional manifolds. In our setting, $\partial_0 A = \partial_1 A =\emptyset$ and $\partial_2 A = \partial A$ is simply the boundary and $\partial_3 A = A^{\circ}$; 
        \item $\mathbb S(T_t\partial_j A^\perp)$ denotes the sphere in the orthogonal complement of the tangent of $\partial_j A$ at the point $t$. When $j=2$, we have the outward and inward normal vectors at the point $t$; whereas when $j=3$, we have $\mathbb S(T_t\partial_j A^\perp)=\{\mathbf 0\}$, the zero vector set;
        \item $R^m$ and $S^l_{\nu}$ denote, respectively, the $m$th power of the Riemann tensor and the $l$th power of the second fundamental form at the vector $\nu\in \mathbb S(T_t\partial_j A^\perp)$, both on $\partial_j A$. We remark the convention that if $\nu = 0$ (zero vector), then $S^l_{\nu}= 1$ if $l=0$ and $0$ otherwise;
        \item $\Tr$ denotes the trace of a double form, see \autoref{sec:doubletrace}, which is a linear operator;
        \item $\mathcal H_{3-j-1}(d \nu)$ and $\mathcal H_j(dt)$ denote, respectively, the volume forms on $\partial_j A$ and $S(T_t\partial_j A^\perp)$;
        \item $\alpha(\nu(t))$ is the normal Morse index at $t$ in the direction $\nu(t)$. We convey that $\alpha(0)=1$. In our setting, $A$ is locally convex and, denoting by $\nu$ the outward normal vector at a point $t\in \partial A$, it holds true that $\alpha(\nu)=0$ and $\alpha(-\nu)=1$; 
        \item $C(m,i) := \frac{2^{i/2-1}\Gamma\left(\frac{m+i}2\right)}{\pi^{m/2}}$ when $m+i>0$, otherwise $1$, see \cite[Eq. (10.5.1)]{AdlerTaylor}.
    \end{itemize}
    Therefore, recalling the convention that $\sum_{m=0}^{\left\lfloor -\frac12 \right\rfloor}=0$, the previous long expression boils down to the following sum \begin{align}
        = &\sum_{m=0}^{\left\lfloor \frac{2-i}2 \right\rfloor} \frac{(-1)^{m-i} \Gamma\left(\frac{3-i-2m}2\right)}{m!(2-i-2m)! 2^{1+m} } \pi^{-(3-i)/2}\int_{\partial A\cap B} \Tr^{T_t\de A}  \left(  R^m  S^{2-i-2m}_{\nu(t)}\right) \mathcal H_2(dt) \\
        & + \sum _{m=0}^{\left\lfloor \frac{3-i}2 \right\rfloor} \frac{(-1)^m (2\pi)^{-(3-i)/2}}{m!(3-i-2m)!} \int_{A^\circ \cap B} \Tr^{T_t A}  \left(  R^m  S^{3-i-2m}_{\mathbf 0}\right) \mathcal H_3(dt)
    \end{align}
    where $\nu(t)$ denotes the outward normal vector at $t\in\partial A$ and $\mathbf 0$ is the zero vector. Recalling the convention for $S_{\mathbf 0}^j$, we obtain the function $R(m,i)$ as defined in \cref{eq:defRmi}.
\end{proof}


\section*{Acknowledgments}
We are thankful to Domenico Marinucci for pointing us at this problem and for his help during our research. 
M.S. is supported by the Luxembourg National Research Fund (Grant: 021/16236290/HDSA). F.P. is partially supported by the Luxembourg National Research Fund (Grant: O22/17372844/FraMStA).

\bibliographystyle{siam}
\bibliography{shortbib}

@article {ASTexcSta,
    AUTHOR = {Adler, Robert J. and Samorodnitsky, Gennady and Taylor,
              Jonathan E.},
     TITLE = {Excursion sets of three classes of stable random fields},
   JOURNAL = {Adv. in Appl. Probab.},
  FJOURNAL = {Advances in Applied Probability},
    VOLUME = {42},
      YEAR = {2010},
    NUMBER = {2},
     PAGES = {293--318},
      ISSN = {0001-8678,1475-6064},
   MRCLASS = {60G52 (60D05 60G10 60G60)},
MRREVIEWER = {Xiaowen\ Zhou},
       DOI = {10.1239/aap/1275055229},
       URL = {https://doi.org/10.1239/aap/1275055229},
}

@book{AdlerTaylor,
   AUTHOR = {Adler, Robert J. and Taylor, Jonathan E.},
     TITLE = {Random fields and geometry},
    SERIES = {Springer Monographs in Mathematics},
 PUBLISHER = {Springer, New York},
      YEAR = {2007},
     PAGES = {xviii+448},
      ISBN = {978-0-387-48112-8},
   MRCLASS = {60G60 (58J65)},
MRREVIEWER = {Jos\'e\ Rafael\ Le\'on},
}

@book {AzaisWscheborbook,
    AUTHOR = {Azais, Jean-Marc and Wschebor, Mario},
     TITLE = {Level sets and extrema of random processes and fields},
 PUBLISHER = {John Wiley \& Sons, Inc., Hoboken, NJ},
      YEAR = {2009},
     PAGES = {xii+393},
      ISBN = {978-0-470-40933-6},
   MRCLASS = {60-02 (60E15 60G05 60G15 60G60 60G70)},
MRREVIEWER = {Anna Amirdjanova},
       DOI = {10.1002/9780470434642},
       URL = {http://dx.doi.org/10.1002/9780470434642},
}

@article {berg_bourguignon,
    AUTHOR = {B\'erard-Bergery, Lionel and Bourguignon, {Jean-Pierre}},
     TITLE = {Laplacians and {R}iemannian submersions with totally geodesic
              fibres},
   JOURNAL = {Illinois J. Math.},
  FJOURNAL = {Illinois Journal of Mathematics},
    VOLUME = {26},
      YEAR = {1982},
    NUMBER = {2},
     PAGES = {181--200},
      ISSN = {0019-2082},
   MRCLASS = {58G25 (53C20)},
MRREVIEWER = {Alan\ West},
}

@article {BB12,
    AUTHOR = {Bobrowski, Omer and Borman, Matthew Strom},
     TITLE = {Euler integration of {G}aussian random fields and persistent
              homology},
   JOURNAL = {J. Topol. Anal.},
  FJOURNAL = {Journal of Topology and Analysis},
    VOLUME = {4},
      YEAR = {2012},
    NUMBER = {1},
     PAGES = {49--70},
      ISSN = {1793-5253,1793-7167},
   MRCLASS = {60G15 (55N35 58C35 60G60)},
MRREVIEWER = {Peter\ Bubenik},
       DOI = {10.1142/S1793525312500057},
       URL = {https://doi.org/10.1142/S1793525312500057},
}

@article {bierme25,
    AUTHOR = {Bierm\'e, Hermine and Desolneux, Agn\`es},
     TITLE = {The anisotropy of 2{D} or 3{D} {G}aussian random fields
              through their {L}ipschitz-{K}illing curvature densities},
   JOURNAL = {Ann. Appl. Probab.},
  FJOURNAL = {The Annals of Applied Probability},
    VOLUME = {35},
      YEAR = {2025},
    NUMBER = {3},
     PAGES = {2031--2079},
      ISSN = {1050-5164,2168-8737},
   MRCLASS = {62M40 (60D05 60G15 60G60 62H11)},
MRREVIEWER = {Mark\ Kelbert}
}

@article {BDiBDE19,
    AUTHOR = {Bierm\'{e}, Hermine and Di Bernardino, Elena and Duval,
              C\'{e}line and Estrade, Anne},
     TITLE = {Lipschitz-{K}illing curvatures of excursion sets for
              two-dimensional random fields},
   JOURNAL = {Electron. J. Stat.},
  FJOURNAL = {Electronic Journal of Statistics},
    VOLUME = {13},
      YEAR = {2019},
    NUMBER = {1},
     PAGES = {536--581},
      ISSN = {1935-7524},
   MRCLASS = {60G60 (60G10 62F03 62F12)},
       DOI = {10.1214/19-EJS1530},
       URL = {https://doi.org/10.1214/19-EJS1530},
}

@article{berger,
    AUTHOR = {Berger, M.},
     TITLE = {Les vari\'et\'es riemanniennes homog\`enes normales simplement
              connexes \`a{} courbure strictement positive},
   JOURNAL = {Ann. Scuola Norm. Sup. Pisa Cl. Sci. (3)},
  FJOURNAL = {Annali della Scuola Normale Superiore di Pisa. Classe di
              Scienze. Serie III},
    VOLUME = {15},
      YEAR = {1961},
     PAGES = {179--246},
      ISSN = {0391-173X},
   MRCLASS = {53.60 (53.72)},
MRREVIEWER = {Yung-Chow\ Wong},
}

@article {BR13,
    AUTHOR = {Baldi, P. and Rossi, M.},
     TITLE = {Representation of {G}aussian isotropic spin random fields},
   JOURNAL = {Stochastic Process. Appl.},
  FJOURNAL = {Stochastic Processes and their Applications},
    VOLUME = {124},
      YEAR = {2014},
    NUMBER = {5},
     PAGES = {1910--1941},
      ISSN = {0304-4149},
   MRCLASS = {60G60 (33C55 57T30)},
MRREVIEWER = {Anatoliy Malyarenko},
       DOI = {10.1016/j.spa.2014.01.007},
       URL = {https://doi.org/10.1016/j.spa.2014.01.007},
}

@ARTICLE{Canzani2020-ty,
   AUTHOR = {Canzani, Yaiza and Hanin, Boris},
     TITLE = {Local universality for zeros and critical points of
              monochromatic random waves},
   JOURNAL = {Comm. Math. Phys.},
  FJOURNAL = {Communications in Mathematical Physics},
    VOLUME = {378},
      YEAR = {2020},
    NUMBER = {3},
     PAGES = {1677--1712},
      ISSN = {0010-3616,1432-0916},
   MRCLASS = {58J50},
MRREVIEWER = {Vassili\ N.\ Kolokol\cprime tsov},
       DOI = {10.1007/s00220-020-03826-w},
       URL = {https://doi.org/10.1007/s00220-020-03826-w},
}

@inproceedings{Cabella:2004mk,
    author = "Cabella, Paolo and Kamionkowski, Marc",
    title = "{Theory of cosmic microwave background polarization}",
    booktitle = "{International School of Gravitation and Cosmology: The Polarization of the Cosmic Microwave Background}",
    eprint = "astro-ph/0403392",
    archivePrefix = "arXiv",
    month = "3",
    year = "2004"
}

@ARTICLE{CMB2023arXiv231200717C,
doi = {10.1088/1475-7516/2024/06/008},
url = {https://doi.org/10.1088/1475-7516/2024/06/008},
year = {2024},
month = {jun},
publisher = {IOP Publishing},
volume = {2024},
number = {06},
pages = {008},
author = {P. {Campeti et al.} and The LiteBIRD collaboration},
title = {LiteBIRD science goals and forecasts. A case study of the origin of primordial gravitational waves using large-scale CMB polarization},
journal = {Journal of Cosmology and Astroparticle Physics},
}

@article {CammarotaM2018,
    AUTHOR = {Cammarota, V. and Marinucci, D.},
     TITLE = {A quantitative central limit theorem for the
              {E}uler-{P}oincar\'{e} characteristic of random spherical
              eigenfunctions},
   JOURNAL = {Ann. Probab.},
  FJOURNAL = {The Annals of Probability},
    VOLUME = {46},
      YEAR = {2018},
    NUMBER = {6},
     PAGES = {3188--3228},
      ISSN = {0091-1798},
   MRCLASS = {60G60 (33C55 42C10 53C65 62M15)},
       DOI = {10.1214/17-AOP1245},
       URL = {https://doi.org/10.1214/17-AOP1245},
}

@article {Camma23,
    AUTHOR = {Cammarota, Valentina and Marinucci, Domenico and Rossi,
              Maurizia},
     TITLE = {Lipschitz-{K}illing curvatures for arithmetic random waves},
   JOURNAL = {Ann. Sc. Norm. Super. Pisa Cl. Sci. (5)},
  FJOURNAL = {Annali della Scuola Normale Superiore di Pisa. Classe di
              Scienze. Serie V},
    VOLUME = {24},
      YEAR = {2023},
    NUMBER = {2},
     PAGES = {1095--1147},
      ISSN = {0391-173X,2036-2145},
   MRCLASS = {60G60 (35P20 58J50 60D05 60F05)},
       DOI = {10.2422/2036-2145.202010\_065},
       URL = {https://doi.org/10.2422/2036-2145.202010_065},
}

@article {CamMarWig16,
    AUTHOR = {Cammarota, V. and Marinucci, D. and Wigman, I.},
     TITLE = {Fluctuations of the {E}uler-{P}oincar\'{e} characteristic for
              random spherical harmonics},
   JOURNAL = {Proc. Amer. Math. Soc.},
  FJOURNAL = {Proceedings of the American Mathematical Society},
    VOLUME = {144},
      YEAR = {2016},
    NUMBER = {11},
     PAGES = {4759--4775},
      ISSN = {0002-9939,1088-6826},
   MRCLASS = {60G60 (33C55 42C10 57R20 60B10 60D05)},
MRREVIEWER = {Kirstin\ Strokorb},
       DOI = {10.1090/proc/13299},
       URL = {https://doi.org/10.1090/proc/13299},
}

@article{articolo_dei_fisici_published,
    AUTHOR = {Carr\'on Duque, J. and Carones, A. and Marinucci, D. and
              Migliaccio, M. and Vittorio, N.},
     TITLE = {Minkowski functionals in {${\rm SO}(3)$} for the spin-2 {CMB}
              polarisation field},
   JOURNAL = {J. Cosmol. Astropart. Phys.},
  FJOURNAL = {Journal of Cosmology and Astroparticle Physics},
      YEAR = {2024},
    NUMBER = {1},
     PAGES = {Paper No. 039, 26},
      ISSN = {1475-7516},
   MRCLASS = {85A35 (83B05 85-04)},
}

@book{Dodelson:2003ft,
    author = "Dodelson, Scott",
    title = "{Modern Cosmology}",
    isbn = "978-0-12-219141-1",
    publisher = "Academic Press",
    address = "Amsterdam",
    year = "2003",
    doi = {10.1016/C2017-0-01943-2},
}

@ARTICLE{RiviereDang2018-gj,
    AUTHOR = {Dang, Nguyen Viet and Rivi\`ere, Gabriel},
     TITLE = {Equidistribution of the conormal cycle of random nodal sets},
   JOURNAL = {J. Eur. Math. Soc. (JEMS)},
  FJOURNAL = {Journal of the European Mathematical Society (JEMS)},
    VOLUME = {20},
      YEAR = {2018},
    NUMBER = {12},
     PAGES = {3017--3071},
      ISSN = {1435-9855,1435-9863},
   MRCLASS = {58J50 (35P20 35R01 58A25 60D05)},
MRREVIEWER = {Leonid\ Friedlander},
       DOI = {10.4171/JEMS/828},
       URL = {https://doi.org/10.4171/JEMS/828},
}

@article {geller2008spin,
    AUTHOR = {Geller, D. and Marinucci, D.},
     TITLE = {Spin wavelets on the sphere},
   JOURNAL = {J. Fourier Anal. Appl.},
  FJOURNAL = {The Journal of Fourier Analysis and Applications},
    VOLUME = {16},
      YEAR = {2010},
    NUMBER = {6},
     PAGES = {840--884},
      ISSN = {1069-5869},
   MRCLASS = {42C40 (33C55 43A85 60G60 83F05)},
MRREVIEWER = {Peter R. Massopust},
       DOI = {10.1007/s00041-010-9128-3},
       URL = {https://doi.org/10.1007/s00041-010-9128-3},
}

@article{Gadea_Oubina_2005, 
    AUTHOR = {Gadea, P. M. and Oubi\~na, J. A.},
     TITLE = {Homogeneous {R}iemannian structures on {B}erger 3-spheres},
   JOURNAL = {Proc. Edinb. Math. Soc. (2)},
  FJOURNAL = {Proceedings of the Edinburgh Mathematical Society. Series II},
    VOLUME = {48},
      YEAR = {2005},
    NUMBER = {2},
     PAGES = {375--387},
      ISSN = {0013-0915,1464-3839},
   MRCLASS = {53C30 (53C20 53D15)},
MRREVIEWER = {Anna\ M.\ Fino},
       DOI = {10.1017/S0013091504000422},
       URL = {https://doi.org/10.1017/S0013091504000422},
}

@article{AHeavens_2008,
doi = {10.1088/1742-6596/120/2/022001},
url = {https://doi.org/10.1088/1742-6596/120/2/022001},
year = {2008},
month = {jul},
publisher = {},
volume = {120},
number = {2},
pages = {022001},
author = {Alan Heavens},
title = {The cosmological model: an overview and an outlook},
journal = {Journal of Physics: Conference Series},
}

@book{Hirsch,
    AUTHOR = {Hirsch, Morris W.},
     TITLE = {Differential topology},
    SERIES = {Graduate Texts in Mathematics},
    VOLUME = {33},
      NOTE = {Corrected reprint of the 1976 original},
 PUBLISHER = {Springer-Verlag, New York},
      YEAR = {1994},
      ISBN = {0-387-90148-5},
   MRCLASS = {57-01 (58-01)},
}

@article{Kuriki_Matsubara_2023, 
    AUTHOR = {Kuriki, Satoshi and Matsubara, Takahiko},
     TITLE = {Asymptotic expansion of the expected {M}inkowski functional
              for isotropic central limit random fields},
   JOURNAL = {Adv. in Appl. Probab.},
  FJOURNAL = {Advances in Applied Probability},
    VOLUME = {55},
      YEAR = {2023},
    NUMBER = {4},
     PAGES = {1390--1414},
      ISSN = {0001-8678,1475-6064},
   MRCLASS = {60G60 (60D05 62M40)},
MRREVIEWER = {Davide\ Giraudo},
       DOI = {10.1017/apr.2023.2},
       URL = {https://doi.org/10.1017/apr.2023.2},
}

@ARTICLE{Nature_Komatsu2022-jv,
    author = "Komatsu, Eiichiro",
    title = "{New physics from the polarized light of the cosmic microwave background}",
    eprint = "2202.13919",
    archivePrefix = "arXiv",
    primaryClass = "astro-ph.CO",
    doi = "10.1038/s42254-022-00452-4",
    journal = "Nature Rev. Phys.",
    volume = "4",
    number = "7",
    pages = "452--469",
    year = "2022"
}

@article {kuwabara,
    AUTHOR = {Kuwabara, Ruishi},
     TITLE = {On spectra of the {L}aplacian on vector bundles},
   JOURNAL = {J. Math. Tokushima Univ.},
  FJOURNAL = {Journal of Mathematics. Tokushima University},
    VOLUME = {16},
      YEAR = {1982},
     PAGES = {1--23},
      ISSN = {0075-4293},
   MRCLASS = {58G25 (35P05 58G30)},
MRREVIEWER = {G\'{e}rard\ Besson},
}

@article {kratzVadlamani,
    AUTHOR = {Kratz, Marie and Vadlamani, Sreekar},
     TITLE = {Central limit theorem for {L}ipschitz-{K}illing curvatures of
              excursion sets of {G}aussian random fields},
   JOURNAL = {J. Theoret. Probab.},
  FJOURNAL = {Journal of Theoretical Probability},
    VOLUME = {31},
      YEAR = {2018},
    NUMBER = {3},
     PAGES = {1729--1758},
      ISSN = {0894-9840,1572-9230},
   MRCLASS = {60F05 (53C65 60D05 60G10 60G15 60G60)},
       DOI = {10.1007/s10959-017-0760-6},
       URL = {https://doi.org/10.1007/s10959-017-0760-6},
}

@book {leeriemann,
    AUTHOR = {Lee, John M.},
     TITLE = {Introduction to {R}iemannian manifolds},
    SERIES = {Graduate Texts in Mathematics},
    VOLUME = {176},
   EDITION = {Second},
 PUBLISHER = {Springer, Cham},
      YEAR = {2018},
     PAGES = {xiii+437},
      ISBN = {978-3-319-91754-2; 978-3-319-91755-9},
   MRCLASS = {53-01 (53B20 53B30 53C20 53C21)},
MRREVIEWER = {Robert\ J.\ Low},
}

@ARTICLE{LiteBIRD_Collaboration2023-jy,
  title     = "Probing cosmic inflation with the {\textit{LiteBIRD}} cosmic
               microwave background polarization survey",
  author    = "{LiteBIRD Collaboration} ",
  journal   = "Prog. Theor. Exp. Phys.",
  publisher = "Oxford University Press (OUP)",
  volume    =  2023,
  number    =  4,
  month     =  apr,
  year      =  2023,
  doi = "10.1093/ptep/ptac150",
  language  = "en"
}

@article {GeoSpin2022,
    AUTHOR = {Lerario, Antonio and Marinucci, Domenico and Rossi, Maurizia
              and Stecconi, Michele},
     TITLE = {Geometry and topology of spin random fields},
   JOURNAL = {Anal. Math. Phys.},
  FJOURNAL = {Analysis and Mathematical Physics},
    VOLUME = {15},
      YEAR = {2025},
    NUMBER = {2},
     PAGES = {Paper No. 48, 70},
      ISSN = {1664-2368,1664-235X},
   MRCLASS = {60G60 (33C47 53C27 85A40)},
MRREVIEWER = {A.\ Ya.\ Olenko},
       DOI = {10.1007/s13324-025-01046-w},
       URL = {https://doi.org/10.1007/s13324-025-01046-w},
}

@article{dtgrf,
    title={Differential Topology of {G}aussian Random Fields},
    author={Lerario, A. and Stecconi, M.},
    JOURNAL = {Preprint ArXiv:1902.03805},
  FJOURNAL = {Preprint ArXiv:1902.03805},
    year={2019},
    eprint={1902.03805},
    archivePrefix={arXiv},
    primaryClass={math.DG}
}

@article {malya11,
    AUTHOR = {Malyarenko, A.},
     TITLE = {Invariant Random Fields in Vector Bundles and Application to Cosmology},
   JOURNAL = {Ann. Inst. Henri Poincar\'{e} Probab. Stat.},
  FJOURNAL = {Annales de l'Institut Henri Poincar\'{e} Probabilit\'{e}s et
              Statistiques},
    VOLUME = {47},
      YEAR = {2011},
    NUMBER = {4},
     PAGES = {1068--1095},
      ISSN = {0246-0203},
   MRCLASS = {60G60 (60B15 85A40)},
MRREVIEWER = {H. Heyer},
       DOI = {10.1214/10-AIHP409},
       URL = {https://doi.org/10.1214/10-AIHP409},
}

@book {malyabook,
    AUTHOR = {Malyarenko, A.},
     TITLE = {Invariant Random Fields on Spaces with a Group Action},
    SERIES = {Probability and its Applications (New York)},
      NOTE = {With a foreword by Nikolai Leonenko},
 PUBLISHER = {Springer, Heidelberg},
      YEAR = {2013},
     PAGES = {xviii+261},
      ISBN = {978-3-642-33405-4; 978-3-642-33406-1},
   MRCLASS = {60G60 (58J65 60G15 60G17 60G22 60G25 62M15 62M40)},
MRREVIEWER = {Yimin Xiao},
       DOI = {10.1007/978-3-642-33406-1},
       URL = {https://doi.org/10.1007/978-3-642-33406-1},
}

@ARTICLE{Malyarenko1999-ay,
    AUTHOR = {Malyarenko, A. A.},
     TITLE = {Local properties of {G}aussian random fields on compact
              symmetric spaces, and {J}ackson-type and {B}ernstein-type
              theorems},
   JOURNAL = {Ukra\"in. Mat. Zh.},
  FJOURNAL = {Nats\=\i onal\cprime na Akadem\=\i ya Nauk Ukra\"ini.
              \=Institut Matematiki. Ukra\"ins\cprime ki\u i\ Matematichni\u
              i\ Zhurnal},
    VOLUME = {51},
      YEAR = {1999},
    NUMBER = {1},
     PAGES = {60--68},
      ISSN = {0041-6053},
   MRCLASS = {60G15 (42B99 43A85 60G17 60G60)},
       DOI = {10.1007/BF02591915},
       URL = {https://doi.org/10.1007/BF02591915},
}

@book{libro,
    AUTHOR = {Marinucci, Domenico and Peccati, Giovanni},
     TITLE = {Random fields on the sphere},
    SERIES = {London Mathematical Society Lecture Note Series},
    VOLUME = {389},
      NOTE = {Representation, limit theorems and cosmological applications},
 PUBLISHER = {Cambridge University Press, Cambridge},
      YEAR = {2011},
     PAGES = {xii+341},
      ISBN = {978-0-521-17561-6},
   MRCLASS = {60G60 (60B15 60D05 60H05 62M15 85A40)},
MRREVIEWER = {Anatoliy\ Malyarenko},
       DOI = {10.1017/CBO9780511751677},
       URL = {https://doi.org/10.1017/CBO9780511751677},
}

@article {MariWig2011,
    AUTHOR = {Marinucci, Domenico and Wigman, Igor},
     TITLE = {On the area of excursion sets of spherical {G}aussian
              eigenfunctions},
   JOURNAL = {J. Math. Phys.},
  FJOURNAL = {Journal of Mathematical Physics},
    VOLUME = {52},
      YEAR = {2011},
    NUMBER = {9},
     PAGES = {093301, 21},
      ISSN = {0022-2488,1089-7658},
   MRCLASS = {60G12 (33C55 60F05 60F17)},
MRREVIEWER = {Dan\ Emanuel\ Popovici},
       DOI = {10.1063/1.3624746},
       URL = {https://doi.org/10.1063/1.3624746},
}

@article{MathiStec,
   AUTHOR = {Mathis, L\'eo and Stecconi, Michele},
     TITLE = {Expectation of a random submanifold: the zonoid section},
   JOURNAL = {Ann. H. Lebesgue},
  FJOURNAL = {Annales Henri Lebesgue},
    VOLUME = {7},
      YEAR = {2024},
     PAGES = {903--967},
      ISSN = {2644-9463},
   MRCLASS = {52A20 (52A22 53C65)},
}

@article {NazarovSodin2016ThatOne,
    AUTHOR = {Nazarov, F. and Sodin, M.},
     TITLE = {Asymptotic laws for the spatial distribution and the number of
              connected components of zero sets of {G}aussian random
              functions},
   JOURNAL = {J. Math. Phys. Anal. Geom.},
  FJOURNAL = {Journal of Mathematical Physics, Analysis, Geometry},
    VOLUME = {12},
      YEAR = {2016},
    NUMBER = {3},
     PAGES = {205--278},
      ISSN = {1812-9471,1817-5805},
   MRCLASS = {60G15},
MRREVIEWER = {Giacomo\ Aletti},
       DOI = {10.15407/mag12.03.205},
       URL = {https://doi.org/10.15407/mag12.03.205},
}

@article {NP66,
    AUTHOR = {Newman, E. T. and Penrose, R.},
     TITLE = {Note on the {B}ondi-{M}etzner-{S}achs group},
   JOURNAL = {J. Mathematical Phys.},
  FJOURNAL = {Journal of Mathematical Physics},
    VOLUME = {7},
      YEAR = {1966},
     PAGES = {863--870},
      ISSN = {0022-2488},
   MRCLASS = {83.22},
MRREVIEWER = {R. K. Sachs},
       DOI = {10.1063/1.1931221},
       URL = {https://doi.org/10.1063/1.1931221},
}

@article {RudnickWigmanTorus,
    AUTHOR = {Rudnick, Ze\'{e}v and Wigman, Igor},
     TITLE = {On the volume of nodal sets for eigenfunctions of the
              {L}aplacian on the torus},
   JOURNAL = {Ann. Henri Poincar\'{e}},
  FJOURNAL = {Annales Henri Poincar\'{e}. A Journal of Theoretical and
              Mathematical Physics},
    VOLUME = {9},
      YEAR = {2008},
    NUMBER = {1},
     PAGES = {109--130},
      ISSN = {1424-0637,1424-0661},
   MRCLASS = {58J50 (42B99 47B80)},
MRREVIEWER = {Nelia\ Charalambous},
       DOI = {10.1007/s00023-007-0352-6},
       URL = {https://doi.org/10.1007/s00023-007-0352-6}
}

@article {SarnakWigman2019,
    AUTHOR = {Sarnak, P. and Wigman, I.},
     TITLE = {Topologies of nodal sets of random band-limited functions},
   JOURNAL = {Comm. Pure Appl. Math.},
  FJOURNAL = {Communications on Pure and Applied Mathematics},
    VOLUME = {72},
      YEAR = {2019},
    NUMBER = {2},
     PAGES = {275--342},
      ISSN = {0010-3640},
   MRCLASS = {58J50 (43A46)},
MRREVIEWER = {Suresh Eswarathasan},
       DOI = {10.1002/cpa.21794},
       URL = {https://doi.org/10.1002/cpa.21794},
}

@article{MikowMorph,
    author = {Schmalzing, Jens and Górski, Krzysztof M.},
    title = "{Minkowski functionals used in the morphological analysis of cosmic microwave background anisotropy maps}",
    journal = {Monthly Notices of the Royal Astronomical Society},
    volume = {297},
    number = {2},
    pages = {355-365},
    year = {1998},
    month = {06},
    issn = {0035-8711},
    doi = {10.1046/j.1365-8711.1998.01467.x},
    url = {https://doi.org/10.1046/j.1365-8711.1998.01467.x},
    eprint = {https://academic.oup.com/mnras/article-pdf/297/2/355/3104443/297-2-355.pdf},
}

@Inbook{Schmalzing1997,
author="Schmalzing, Jens
and Kerscher, Martin",
title="Minkowski Functionals in Cosmology",
bookTitle="Generation of Cosmological Large-Scale Structure",
year="1997",
publisher="Springer Netherlands",
address="Dordrecht",
pages="255--260",
isbn="978-94-009-0053-0"
}

@ARTICLE{Seljak,
       author = {{Seljak}, Uros and {Zaldarriaga}, Matias},
        title = "{A Line-of-Sight Integration Approach to Cosmic Microwave Background Anisotropies}",
      journal = {The Astrophysical Journal},
         year = 1996,
        month = oct,
       volume = {469},
        pages = {437},
          doi = {10.1086/177793},
archivePrefix = {arXiv},
       eprint = {astro-ph/9603033},
 primaryClass = {astro-ph},
       adsurl = {https://ui.adsabs.harvard.edu/abs/1996ApJ...469..437S}
}

@article{stecconi2021isotropic,
    AUTHOR = {Stecconi, Michele},
     TITLE = {Isotropic random spin weighted functions on {$S^2$} vs
              isotropic random fields on {$S^3$}},
   JOURNAL = {Theory Probab. Math. Statist.},
  FJOURNAL = {Theory of Probability and Mathematical Statistics},
    NUMBER = {107},
      YEAR = {2022},
     PAGES = {77--109},
      ISSN = {0094-9000,1547-7363},
   MRCLASS = {60G60 (20C35 53C20 60B20 81S30)},
       DOI = {10.1090/tpms/1177},
       URL = {https://doi.org/10.1090/tpms/1177},
}

@article{stecconi2021kacrice,
    AUTHOR = {Stecconi, Michele},
     TITLE = {Kac-{R}ice formula for transverse intersections},
   JOURNAL = {Anal. Math. Phys.},
  FJOURNAL = {Analysis and Mathematical Physics},
    VOLUME = {12},
      YEAR = {2022},
    NUMBER = {2},
     PAGES = {Paper No. 44, 64},
      ISSN = {1664-2368,1664-235X},
   MRCLASS = {60D05 (28C20 57N75 58K05 60G15)},
       DOI = {10.1007/s13324-022-00654-0},
       URL = {https://doi.org/10.1007/s13324-022-00654-0},
}

@article{urakawa,
    AUTHOR = {Urakawa, Hajime},
     TITLE = {On the least positive eigenvalue of the {L}aplacian for
              {R}iemannian manifolds},
   JOURNAL = {Proc. Japan Acad. Ser. A Math. Sci.},
  FJOURNAL = {Japan Academy. Proceedings. Series A. Mathematical Sciences},
    VOLUME = {53},
      YEAR = {1977},
    NUMBER = {7},
     PAGES = {229--231},
      ISSN = {0386-2194},
   MRCLASS = {58E05 (53C20 58G99)},
MRREVIEWER = {J.\ S.\ Joel},
       URL = {http://projecteuclid.org/euclid.pja/1195517960},
}

@article {VidLipschitz,
    AUTHOR = {Vidotto, A.},
     TITLE = {Random {L}ipschitz-killing curvatures: reduction principles,
              integration by parts and {W}iener chaos},
   JOURNAL = {Theory Probab. Math. Statist.},
  FJOURNAL = {Theory of Probability and Mathematical Statistics},

      YEAR = {2022},
     PAGES = {157--175},
      ISSN = {0094-9000,1547-7363},
   MRCLASS = {60-02 (35J05 60D05 60G10 60G15 60G60)},
       DOI = {10.1090/tpms/1170},
       URL = {https://doi.org/10.1090/tpms/1170},
}

@incollection {zelditch_2009_rczRRW,
    AUTHOR = {Zelditch, Steve},
     TITLE = {Real and complex zeros of {R}iemannian random waves},
 BOOKTITLE = {Spectral analysis in geometry and number theory},
    SERIES = {Contemp. Math.},
    VOLUME = {484},
     PAGES = {321--342},
 PUBLISHER = {Amer. Math. Soc., Providence, RI},
      YEAR = {2009},
      ISBN = {978-0-8218-4269-0},
   MRCLASS = {58J51 (35J05)},
MRREVIEWER = {Julie\ Rowlett},
       DOI = {10.1090/conm/484/09482},
       URL = {https://doi.org/10.1090/conm/484/09482},
}

\end{document}